\UseRawInputEncoding
\documentclass[12pt, reqno]{amsart}
\usepackage[margin=1in]{geometry}
\usepackage{amssymb,latexsym,amsmath,amscd,amsfonts}
\usepackage{latexsym}
\usepackage[mathscr]{eucal}
\usepackage{bm}
\usepackage{graphicx}
\usepackage{forest}
\usepackage{tikz-qtree}
\usepackage{mathptmx}
\usepackage{amssymb}
\usepackage{upgreek}
\usepackage{amsthm}
\usepackage{dcolumn}
\usepackage{tikz-cd}
\usepackage[all]{xy}
\usepackage{enumitem}
\usepackage[utf8]{inputenc}

\def \qed {\hfill \vrule height6pt width 6pt depth 0pt}
\def\textmatrix#1&#2\\#3&#4\\{\bigl({#1 \atop #3}\ {#2 \atop #4}\bigr)}
\def\dispmatrix#1&#2\\#3&#4\\{\left({#1 \atop #3}\ {#2 \atop #4}\right)}
\newcommand{\beg}{\begin{equation}}
	\newcommand{\eeg}{\end{equation}}
\newcommand{\ben}{\begin{eqnarray*}}
	\newcommand{\een}{\end{eqnarray*}}

\newtheorem{thm}{Theorem}[section]
\newtheorem{cor}[thm]{Corollary}
\newtheorem{lem}[thm]{Lemma}

\newtheorem{prop}[thm]{Proposition}
\numberwithin{equation}{section} \theoremstyle{definition}
\newtheorem{defn}[thm]{Definition}
\newtheorem{rem}[thm]{Remark}

\newtheorem{eg}[thm]{Example}

\newcommand{\C}{\mathbb{C}}

\newcommand{\D}{\mathbb{D}}

\newcommand{\T}{\mathbb{T}}
\newcommand{\N}{\mathbb{N}}

\newcommand{\HS}{\mathcal{H}}

\def\textmatrix#1&#2\\#3&#4\\{\bigl({#1 \atop #3}\ {#2 \atop #4}\bigr)}
\def\dispmatrix#1&#2\\#3&#4\\{\left({#1 \atop #3}\ {#2 \atop #4}\right)}

\title[Decompositions of $Q$-commuting contractions]{Theory of $Q$-commuting contractions: joint reducing subspaces and orthogonal decompositions}

\author[PAL, SAHASRABUDDHE AND TOMAR]{SOURAV PAL, PRAJAKTA SAHASRABUDDHE AND NITIN TOMAR}

\address[Sourav Pal]{Mathematics Department, Indian Institute of Technology Bombay,
	Powai, Mumbai - 400076, India.} \email{sourav@math.iitb.ac.in}

\address[Prajakta Sahasrabuddhe]{Mathematics Department, Indian Institute of Technology Bombay, Powai, Mumbai-400076, India.} \email{prajakta@math.iitb.ac.in}	

\address[Nitin Tomar]{Mathematics Department, Indian Institute of Technology Bombay, Powai, Mumbai-400076, India.} \email{tnitin@math.iitb.ac.in}	



\begin{document}

\maketitle

\begin{abstract}
A family $\mathcal{T}=\{T_i: i \in J\}$ of operators acting on a Hilbert space $\HS$ is said be \textit{$Q$-commuting} if there exists a commuting family $Q=\{Q_{ij} \in \mathcal{B}(\HS) \ : \ Q_{ij}=Q_{ji}^*, i \ne j \in J \}$ of unitaries such that $T_iT_j=Q_{ij}T_jT_i$ and $T_kQ_{ij}=Q_{ij}T_k$ for all $i, j, k$ in $J$ with $i \ne j$. The family $\mathcal{T}$ is said to be \textit{doubly $Q$-commuting} if it is $Q$-commuting and $T_i T_j^*=Q_{ij}^*T_j^*T_i$ for all $i, j$ in $J$ with $i \ne j$. In this article, we obtain the following main results.

\begin{enumerate}
\item[(i)] It is well known that any number of doubly commuting isometries admit a Wold-type decomposition. Also, examples from the literature show that such decomposition does not hold in general for commuting isometries. We generalize this result to any doubly $Q$-commuting family of contractions and get a canonical decomposition for them. Then we obtain a Wold-type decomposition for any doubly $Q$-commuting family of isometries as a special case.

\smallskip 

\item[(ii)] Further, we show that a similar decomposition is possible for an arbitrary doubly $Q$-commuting family of c.n.u. contractions, where all members jointly orthogonally decompose into pure isometries (i.e. unilateral shifts) and completely non-isometry (c.n.i.) contractions. 

\smallskip

\item[(iii)] We present a decomposition result for $Q$-commuting tuples of contractions which is analogous to Z. Burdak's work on commuting contractions.

\smallskip

\item[(iv)] Also, we provide decomposition results for a pair of contractions $(T_1, T_2)$ satisfying any of the following relations:
\[
T_1 T_2 = QT_2 T_1, \quad T_1 T_2 = T_2 Q T_1,  \quad T_1 T_2 = T_2 T_1 Q,
\]
where $Q$ is any unitary that commutes with the product $T_1 T_2$.
\end{enumerate}		
	\end{abstract} 
	
	\section{Introduction}\label{sec01}
	
		\vspace{0.2cm}
	
	\noindent Throughout the paper, all operators are bounded linear maps defined on a complex Hilbert space. For two Hilbert spaces $\HS_1, \HS_2$, the algebra of operators from $\HS_1$ to $\HS_2$ is denoted by $\mathcal B(\HS_1, \HS_2)$ and the notation $\mathcal B(\HS)$ is used when $\HS_1=\HS_2=\HS$. A \textit{contraction} is an operator with norm at most $1$. For a contraction $T$ on a Hilbert space $\HS$, $D_T=(I-T^*T)^{1\slash 2}$ denotes the positive square root of $I-T^*T$ and $\mathcal{D}_T=\overline{D_T\HS}$. We denote by $\mathbb{C}, \mathbb{D}$ and $\mathbb{T}$ the complex plane, the unit disk and the unit circle in the complex plane, respectively, with center at the origin.
	
	\smallskip
	
Determining the common reducing subspaces of a family of commuting operators has always been fascinating to the operator theorists. The first step in this direction was the famous von Neumann-Wold decomposition of an isometry. 
	\begin{thm}[von Neumann-Wold, \cite{Wold}]\label{thm_Wold}
		Let $V$ be an isometry on a Hilbert space $\mathcal{H}$. Then there is a decomposition of $\mathcal{H}$ into an orthogonal sum of two subspaces reducing $V$, say $\mathcal{H}=\mathcal{H}_0 \oplus \mathcal{H}_1$ such that $V|_{\mathcal{H}_{0}}$ is a unitary and $V|_{\mathcal{H}_1}$ is a pure isometry, i.e. a unilateral shift. This decomposition is uniquely determined. Indeed, we have 
		\[
		\HS_0=\bigcap_{k=0}^\infty V^k\HS \quad \text{and} \quad \HS_1=\overset{\infty}{\underset{k=0}{\oplus}}V^k\mathcal{L} \quad \text{where} \quad \mathcal{L}=\HS \ominus V\HS.
		\]
	\end{thm}	

The space $\mathcal{H}_0$ in this decomposition is the maximal reducing subspace on which $V$ acts as a unitary. Naturally, the idea of separating the maximal unitary part from an isometry was further extended to any contraction which is widely known as the \textit{canonical decomposition} of a contraction. 	
		
	\begin{thm}[\cite{Nagy}, Chapter I, Theorem 3.2]\label{thm104}
		For every  contraction $T$ on a Hilbert space $\mathcal{H}$, there is a unique  decomposition of $\mathcal{H}$ into an orthogonal sum of two subspaces reducing $T$, say $\mathcal{H}=\mathcal{H}_0 \oplus \mathcal{H}_1$, such that $T|_{\mathcal{H}_0}$ is a unitary and $T|_{\mathcal{H}_1}$ is a completely non-unitary (c.n.u.) contraction. Either of the subspaces $\mathcal{H}_0$ or $\mathcal{H}_1$ may coincide with the zero subspace.  Indeed, $\mathcal{H}_0$ is the set of all elements in $\mathcal{H}$ such that 
		\[
		\|T^nh\|=\|h\|=\|T^{*n}h\|, \quad n=1,2, \dotsc \ .
		\]  
	\end{thm}

In \cite{Levan}, Levan went one step ahead and managed to decompose a c.n.u. (completely non-unitary) contraction further into two orthogonal parts.
	
\begin{thm}[\cite{Levan}, Theorem 1]\label{thm103}
		Let $T$ be a c.n.u. contraction acting on a Hilbert space $\mathcal{H}$. Then $\mathcal{H}$ decomposes into an orthogonal sum $\mathcal{H}=\mathcal{H}_1\oplus \mathcal{H}_2$ such that $\mathcal{H}_1, \mathcal{H}_2$ reduce $T$ and that $T|_{\mathcal{H}_1}$ is an isometry and $T|_{\mathcal{H}_2}$ is a completely non-isometry (c.n.i.) contraction.
	\end{thm}

There are various analogues of Wold decomposition of an isometry in several variable setting, e.g. \cite{BCL, BurdakI, BurdakII, Popovici, Slocinski} and an appealing among them is due to S\l oci\'{n}ski. 
	
	\begin{thm}[M. S\l oci\'{n}ski, \cite{Slocinski}]\label{thm_Slo}
		Let $(V_1, V_2)$ be a doubly commuting pair of isometries on a Hilbert space $\mathcal{H}$. Then there is a unique decomposition 
		$
			\mathcal{H}=\mathcal{H}_{uu} \oplus \mathcal{H}_{us} \oplus \mathcal{H}_{su} \oplus  \mathcal{H}_{ss},
	$ where $\mathcal{H}_{uu}, \mathcal{H}_{us},\mathcal{H}_{su},\mathcal{H}_{ss}$ are closed subspaces of $\mathcal{H}$ reducing $V_1$ and $V_2$
		such that 
		\begin{enumerate}
			\item[(a)] $V_1|_{\mathcal{H}_{uu}}$ and $V_2|_{\mathcal{H}_{uu}}$ are unitaries; 
			\item[(b)] $V_1|_{\mathcal{H}_{us}}$ is a unitary and $V_2|_{\mathcal{H}_{us}}$ is a shift;
			\item[(c)] $V_1|_{\mathcal{H}_{su}}$ is a shift and $V_2|_{\mathcal{H}_{su}}$ is a unitary; 
			\item[(d)] $V_1|_{\mathcal{H}_{ss}}$ and $V_2|_{\mathcal{H}_{ss}}$ are shifts.
		\end{enumerate}
		Some of the spaces among $\mathcal{H}_{uu}, \mathcal{H}_{us}, \mathcal{H}_{su}$ and  $\mathcal{H}_{ss}$ may coincide with the trivial space $\{0\}$.
	\end{thm}

	 S\l oci\'{n}ski's decomposition is complete in the sense that it leaves no scope of decomposing the original operator pair any further. However, such a decomposition is not possible for a pair of commuting isometries as was shown in \cite{BurdakII, Slocinski}. Interestingly, Popovici \cite{Popovici} provided a new way of decomposing a pair of commuting isometries. S\l oci\'{n}ski's decomposition was generalized in \cite{BurdakII} to any family of doubly commuting isometries with finite dimensional wandering subspaces. Additionally, there are canonical decomposition for commuting contractions in various directions and under different conditions, e.g. \cite{Albrecht, BCL, BurdakI, BurdakIII,  Cate, Popovici}. Eschmeier's work \cite{Esch} contributes significantly to this area. In \cite{Pal}, the first named author of this paper generalized these decomposition results to any family of doubly commuting contractions using a new technique. These compelling works motivate us to study the joint reducing subspaces and orthogonal decompositions for operators that satisfy $Q$-commuting relations.

\begin{defn}
		A family $\mathcal{T}=\{T_i: i \in J\}$ of operators acting on a  Hilbert space $\HS$ is said be \textit{$q$-commuting} if there is a family of non-zero scalars $q=\{q_{ij} \in \C : q_{ij}=q_{ji}^{-1}, i \ne j \in J \}$ such that $T_iT_j=q_{ij}T_jT_i$ for all $i, j \in J$ with $i \ne j$. We say that $\mathcal{T}$ is a \textit{$Q$-commuting} family if there exists a commuting family $Q=\{Q_{ij} \in \mathcal{B}(\HS) \ : \ Q_{ij}=Q_{ji}^*,  \ i \ne j \in J \}$ of unitaries such that 
		$
		T_i T_j =Q_{ij}T_j T_i$ and $T_kQ_{ij}=Q_{ij}T_k$
		for all $i \ne j$ and $k$ in $J$.  We simply say that $\mathcal{T}$ is a $Q$-commuting family without referring the family $Q$. If each $Q_{ij}=q_{ij}I$ with $|q_{ij}|=1$, we say that $\mathcal{T}$ is a \textit{$q$-commuting} family with $\|q\|=1$.
	\end{defn}
	
\begin{defn}\label{defn_dc_q} 
		A family $\mathcal{T}=\{T_i: i \in J\}$ of operators acting on a  Hilbert space $\HS$ is said be \textit{doubly $Q$-commuting} if there is a commuting family $Q=\{Q_{ij} \in \mathcal{B}(\HS) \ : \ Q_{ij}=Q_{ji}^*, \   i \ne j \in J \}$ of unitaries such that 
		$
		T_i T_j =Q_{ij}T_j T_i, \, T_i T_j^*=Q_{ij}^*T_j^*T_i$ and $T_kQ_{ij}=Q_{ij}T_k$
		for all $i \ne j$ and $k$ in $J$. We say that $\mathcal T$ is a doubly $Q$-commuting family without mentioning the family $Q$. If each $Q_{ij}=q_{ij}I$ with $|q_{ij}|=1$, we say that $\mathcal{T}$ is a \textit{doubly $q$-commuting} family.
\end{defn}

Let $\mathcal T$ be a $Q$-commuting family of operators acting on a Hilbert space $\HS$. A closed linear subspace $\mathcal{L} \subseteq \HS$ is said to be a \textit{joint reducing subspace} for  $\mathcal T$ and $Q$ if $\mathcal{L}$ is a reducing subspace for each member of $\mathcal T$ and $Q$. 

\vspace{0.1cm}

In this paper, we generalize the classical decomposition results to any family of $Q$-commuting and doubly $Q$-commuting contractions in different ways. In Section \ref{Prep section}, we define doubly $q$-commuting contractions appropriately and describe the motivation behind defining this way. Also, there we prove a few preparatory results. In Sections \ref{sec03} \& \ref{sec04}, we generalize the decomposition results in Theorems \ref{thm_Wold}, \ref{thm104}, \ref{thm103} \& \ref{thm_Slo} to any family of doubly $Q$-commuting contractions. Our first main result in this context is a fine canonical decomposition for any doubly $Q$-commuting family of contractions. In Theorem  \ref{thm303}, we show that for any doubly $Q$-commuting family of contractions $\mathcal{T}=\{T_i : i \in J\}$ acting on a Hilbert space $\mathcal{H},$ the space $\mathcal{H}$ continues to decompose orthogonally into joint reducing subspaces of $\mathcal{T}$ until all members of $\mathcal{T}$ split together into unitaries and c.n.u. contractions. Consequently, we obtain a von Neumann-Wold decomposition for any doubly $Q$-commuting family of isometries in Theorem \ref{Wold general case}. In Section \ref{sec04}, we find a Levan-type decomposition (as in Theorem \ref{thm103}) for any family of doubly $Q$-commuting c.n.u. contractions, where all members admit a simultaneous decomposition into isometries and c.n.i. contractions. 
	
	\vspace{0.1cm}
	
 The decomposition results achieved so far are based on the bound of double $Q$-commutativity. In Theorem \ref{thm501}, we present a decomposition of a different kind for a finite tuple of $Q$-commuting contractions which generalizes Burdak's work \cite{BurdakIII} for a pair of commuting contractions. As a special case, we have a Wold-type decomposition for finitely many $Q$-commuting isometries in Theorem \ref{prop503}. In Section \ref{sec07}, we consider a pair of contractions $(T_1, T_2)$ for which there is a unitary $Q$  such that any one of $T_1T_2=QT_2T_1$, $T_1T_2=T_2QT_1$, $T_1T_2=T_2T_1Q$ holds. We prove a decomposition result in Theorem \ref{thm712} for such a pair of contractions under the weaker hypothesis that $Q$ commutes with the product $T_1T_2$.
 
\vspace{0.1cm}

\noindent \textbf{Note.} After writing the first draft of this paper, we learned that Theorem \ref{thm303F} of this article is independently proved by Majee and Maji as Theorem 4.2 in \cite{Maji}. However, our proof to this result here is based on different techniques.

\vspace{0.1cm}
	
	\section{Basic terminologies and preparatory results }\label{Prep section}
	
		\vspace{0.2cm}
	
\noindent  We recall from the literature a few significant classes of contractions and introduce some new terminologies useful for our purposes.
	
	\begin{defn} A contraction $T$ acting on a Hilbert space $\mathcal{H}$ is called
		\begin{enumerate}
			\item a \textit{completely non-unitary}, or simply a \textit{c.n.u.}, if there is no closed subspace of $\mathcal{H}$ that reduces $T$ and on which $T$ acts as a unitary.
			\item a \textit{completely non-isometry}, or simply a \textit{c.n.i.}, if there is no closed subspace of $\mathcal{H}$ that reduces $T$ and on which $T$ acts as an isometry.
			\item a \textit{pure contraction}, or a $C._0$ \textit{contraction}, if ${T^{*}}^nx \to 0$ as $n \to \infty$ for every $x \in \mathcal{H}.$
			\item a \textit{pure isometry}, or a \textit{unilateral shift}, if $T$ is an isometry which is also a pure contraction. 
		\end{enumerate}	
	\end{defn}
	
Let us recall that a family $\mathcal{T}=\{T_i : i \in J\}$ of operators is called $q$-commuting with $\|q\|=1$ if there is a set $q=\{q_{ij} \in \T :\, q_{ij}= \overline{q}_{ji},\, i \ne j \in J \}$ such that $T_iT_j=q_{ij}T_jT_i$ for all $ i, j \in J$ with $i \ne j$. In addition, if $T_iT_j^*=\overline{q}_{ij}T_j^*T_i$ for $i, j \in J$ with $i \ne j$, then $\mathcal{T}$ is called doubly $q$-commuting whereas $\mathcal T$ is doubly commuting if $T_iT_j^*=T_j^*T_i$ for all $i,j$ in $J$ with $i \ne j$. The following observations justify the reason behind defining doubly $q$-commuting family in this manner.  
	
\begin{enumerate}

			\item Let $T_1, T_2$ be a $q$-commuting pair of operators and let $T_2$ be a unitary. Then
			\begin{equation*}
				T_1T_2=qT_2T_1 \implies T_2^*T_1=qT_1T_2^* \implies T_1T_2^*=q^{-1}T_2^*T_1.
			\end{equation*}
			Furthermore, $T_1=T_2(qT_1)T_2^*$ implying that $T_1$ and $qT_1$ are unitarily equivalent. Consequently, $\|T_1\|=\|qT_1\|$ which holds if and only if $|q|=1$ unless $T_1$ is zero. So, if $T_1 \neq 0$, then we must have $T_1T_2^*=\overline{q}T_2^*T_1$ and $|q|=1$.
			
			\smallskip
			
\item Let us consider a $q$-commuting pair of $n \times n$ matrices $(T, N)$ such that $N$ is normal. We can assume using the spectral theorem that $N$ is a diagonal matrix with respect to some fixed orthonormal basis and so, we have that
			\[
			T=\begin{bmatrix}
				a_{11}& a_{12} & \dotsc & a_{1n}\\
				a_{21} & a_{22} & \dotsc & a_{2n}\\
				\dotsc & \dotsc & \ddots & \dotsc \\
				a_{n1} & a_{n2} & \dotsc & a_{nn}\\
			\end{bmatrix} \quad \mbox{and} \quad 
			N=\begin{bmatrix}
				\lambda_{1}& 0 & \dotsc & 0\\
				0 & \lambda_2 & \dotsc & 0\\
				\dotsc & \dotsc & \ddots & \dotsc \\
				0 & 0 & \dotsc & \lambda_n \\
			\end{bmatrix}.
			\]
			It follows from some routine calculations that 
			\[
			TN-qNT=\left[(\lambda_j-q\lambda_i)a_{ij}\right]_{i, j=1}^n \quad \text{and} \quad TN^*-\overline{q}N^*T=\left[\overline{(\lambda_j-q\lambda_i)}a_{ij}\right]_{i, j=1}^n.
			\]
			Since $TN=qNT$, we must have $(\lambda_j-q\lambda_i)a_{ij}=0$ and so, $\overline{(\lambda_j-q\lambda_i)}a_{ij}=0$ for each $i, j$. Therefore, $TN^*=\overline{q}N^*T.$
		\end{enumerate}	
	
It is not difficult to provide a generalization of the above remarks using the following result.
	
\begin{thm}[Fuglede-Putnam, \cite{Fuglede, Putnam}]\label{Putnam}
		Let $T, M$ and $N$ be operators acting on a Hilbert space $\HS$. If $M$ and $N$ are normal such that $MT=TN$, then $M^*T=TN^*$.
	\end{thm}

	\begin{prop}\label{normal}
		Let $(T_1, T_2)$ be a $q$-commuting pair of operators acting on a Hilbert space $\HS$. Then the following holds. 
		\begin{enumerate}
			\item[(i)] If $T_1$ is normal, then $T_1T_2^*=q^{-1}T_2^*T_1;$
			\item[(ii)] If $T_2$ is normal, then $T_1T_2^*=\overline{q}T_2^*T_1;$
			\item[(iii)] If $T_1$ and $T_2$ are normal with $T_1^*T_2 \neq 0$, then $|q|=1$.
		\end{enumerate}
	\end{prop}	
	
\begin{proof} The proof is an immediate consequence of Theorem \ref{Putnam}.
		\begin{enumerate}
			\item[(i)]  Let $T_1$ be normal and let $M=T_1, N=qT_1$ and $T=T_2$. By $q$-commutativity condition, we have that $MT=TN$. It follows from Theorem \ref{Putnam} that $T^*M=NT^*$ and so, $T_2^*T_1=qT_1T_2^*$.
			
\item[(ii)]  Assume that $T_2$ is normal. Define $M=qT_2, N=T_2$ and $T=T_1$. We have by Theorem \ref{Putnam} that $M^*T=TN^*$ and thus, $\overline{q}T_2^*T_1=T_1T_2^*$.
			
\item[(iii)] Let $T_1$ and $T_2$ be normal operators. It follows from the previous parts that $T_1T_2^*=\overline{q}T_2^*T_1=q^{-1}T_2^*T_1$. Therefore, $|q|=1$ if $T_1^*T_2 \ne 0$. The proof is complete.
		\end{enumerate} 
	\end{proof}
	
Needless to mention that Proposition \ref{normal} holds if $q$ is a unitary operator that commutes with $T_1, T_2$. Also, this result reveals the reason behind choosing unimodular scalars $q$ while defining doubly $q$-commuting operators as in Definition \ref{defn_dc_q}. An analogue of Theorem \ref{Putnam} holds trivially for $q$-commuting normal operators as a consequence of Proposition \ref{normal}.
	
\begin{cor}\label{normal_DC}
		Any $q$-commuting family of normal operators with $\|q\|=1$ is doubly $q$-commuting.
	\end{cor}
	
 Indeed, we prove a stronger version of the above result. We prove that given a $q$-commuting family of normal operators $\{N_i : i \in J \}$ with $q=\{q_{ij} : i \ne j \in J \}$, one can choose a different set of scalars, say $\widetilde{q}=\{\widetilde{q}_{ij} : i \ne j \in J\}$, such that $|\widetilde{q}_{ij}|=1$ and $\{N_i : i \in J\}$ is doubly $\widetilde{q}$-commuting. 
	
	\begin{lem}\label{thm207}
		Any $q$-commuting family of normal operators is a doubly $\widetilde{q}$-commuting family.
	\end{lem}
	
	\begin{proof}
		Let $\{N_i : i \in J\}$ be a $q$-commuting family of normal operators. Then $N_iN_j=q_{ij}N_jN_i$ and $q_{ij} \ne 0$ for each $i \ne j$ in $J$. Take any $i \ne j$ in $J$. It follows from Corollary \ref{normal_DC} and Proposition \ref{normal} that   
		\begin{equation}\label{eqn101}
			N_iN_j^*=\overline{q}_{ij}N_j^*N_i
		\end{equation}
and $|q_{ij}|=1$ if $N_i^*N_j \ne 0$. Let $N_i^*N_j=0$. We show that $N_iN_j=0$. We have by (\ref{eqn101}) that
		\begin{equation*}
			\begin{split}
				(N_iN_j)(N_iN_j)^*=q_{ij}N_j(N_iN_j^*)N_i^*
				=q_{ij}\overline{q}_{ij}N_jN_j^*N_iN_i^* 				=|q_{ij}|^2N_j(N_i^*N_j)^*N_i^*
				=0.
			\end{split}
		\end{equation*}
		Thus, $N_iN_j=0$. We define 
		$
		\widetilde{q}_{ij}:= \left\{
		\begin{array}{ll}
			q_{ij}, & N_i^*N_j \ne 0 \\
			1, & N_i^*N_j = 0 \\
		\end{array} 
		\right. 
		$
		for every $i \ne j$ in $J$. Consequently, 
		\[
		N_iN_j=\widetilde{q}_{ij}N_jN_i, \quad N_i^*N_j=\widetilde{q}_{ij}^{-1}N_jN_i^* \quad \text{and} \quad |\widetilde{q}_{ij}|=1.
		\]
\end{proof}
	
The following examples from \cite{BPP, K.M.} show that the classes of $q$-commuting contractions with $\|q\|=1$ and doubly $q$-commuting contractions are non-empty. An interested reader is also referred to \cite{Maji} for more examples of $Q$-commuting contractions.

\begin{eg}
Given a Hilbert space $\mathcal{E}$ , let $\ell^2(\mathcal E)$ be the space of square summable sequences with enteries from $\mathcal E$. Consider the operators $T_1$ and $T_2$ on $\ell^2(\mathcal E)$ given by
\[
T_1(a_0, a_1, a_2, \dotsc) = (a_0, qa_1, q^2a_2, \dotsc) \quad \text{and} \quad  T_2(a_0, a_1, a_2, \dotsc) = (0, a_0, a_1, \dotsc) \quad (|q| = 1).
\]
It is easy to verify that $T_1$ and $T_2$ are isometries such that $T_1T_2 = qT_2T_1$. Moreover,
\[
T_1^*(a_0, a_1, a_2, \dotsc ) = (a_0, \overline{q}a_1, \overline{q}^2a_2, \dotsc)  \quad \text{for} \quad (a_0, a_1, a_2, \dotsc) \in \ell^2(\mathcal E ).
\]
It is a simple exercise to show that $T_1T_2^*=\overline{q}T_2^*T_1$ and so, $(T_1, T_2)$ is doubly $q$-commuting. \qed
\end{eg}

\begin{eg}
For $|q| = 1$ and non-zero scalars $a_1, a_2$ and $b$, let us define
\[
A=\begin{pmatrix}
a_1 & 0 \\
a_2 & qa_1
\end{pmatrix} \quad \text{and} \quad B=\begin{pmatrix}
0 & 0\\
b & 0
\end{pmatrix}.
\]
Some routine calculations show that $AB = qBA$ but $AB^*\ne \overline{q}B^*A$. Therefore, $(\|A\|^{-1}A, \|B\|^{-1}B)$ is a $q$-commuting pair of contractions which is not doubly $q$-commuting.	\qed
\end{eg}	
	

\section{Decomposition of any family of doubly $Q$-commuting contractions}\label{sec03}
	
		\vspace{0.2cm}
	
	\noindent In this Section, we present a canonical decomposition of any doubly $Q$-commuting family of contractions. The next theorem plays a central role to obtain the desired decomposition result. A reader is referred to \cite{Maji} for a proof to this. However, we give an alternative proof here.
	
	\begin{thm}\label{Qreducing}
		For a doubly $Q$-commuting pair of contractions $A$ and $B$ acting on a Hilbert space $\mathcal{H}$, if $B=B_1\oplus B_2$ is the canonical decomposition of $B$ with respect to the orthogonal decomposition $\mathcal{H}=\mathcal{H}_1 \oplus \mathcal{H}_2,$ then $\mathcal{H}_1, \mathcal{H}_2$ are reducing subspaces for $A$ and $Q$.
	\end{thm}
	
	\begin{proof}
To avoid the trivial cases, we assume that $\HS_1$ and $\HS_2$ both are non-zero subspaces. Since $Q$ is a normal operator which commutes with $A$ and $B$, it follows from Fuglede's Theorem \cite{Fuglede} that $Q$ doubly commutes with $A$ and $B$. We have by Theorem \ref{thm104} that
\[
\HS_1=\{h \in \HS : \|B^nh\|=\|h\|=\|{B^*}^nh\| \ \ \text{for} \ \ n=1, 2, \dotsc \} \quad \text{and} \quad \HS_2=\HS \ominus \HS_1.
\]
Let $h \in \HS_1$. Since $Q$ is a unitary that doubly commutes with $B$, we have that 
\[
\|B^nQh\|=\|QB^nh\|=\|B^nh\|=\|h\|=\|Qh\| \quad \text{and} \quad \|{B^*}^nQh\|=\|Q{B^*}^nh\|=\|{B^*}^nh\|=\|h\|=\|Qh\|
\]
for every non-negative integer $n$. Thus, $\HS_1$ and $\HS_2$ are reducing subspaces for $Q$. With respect to $\HS=\HS_1 \oplus \HS_2$, let us assume that
		\[
		A=\begin{bmatrix}
			A_{11} & A_{12} \\
			A_{21} & A_{22}		
		\end{bmatrix}, \quad B=\begin{bmatrix}
			B_1 & 0 \\
			0 & B_2		
		\end{bmatrix} \quad \text{and} \quad Q=\begin{bmatrix}
			Q_1 & 0\\
			0 & Q_2
		\end{bmatrix}. 
		\]
Since $AQ=QA$, we have that 
		\begin{align}
			A_{11}Q_1&=Q_1A_{11}, \quad A_{12}Q_2=Q_1A_{12}, \label{eqn601}\\
			A_{21}Q_1&=Q_2A_{21}, \quad A_{22}Q_2=Q_2A_{22}. \label{eqn602}
		\end{align}
		Moreover, $AQ^*=Q^*A$ gives that 
		\begin{align}
			A_{11}Q_1^*&=Q_1^*A_{11}, \quad A_{12}Q_2^*=Q_1^*A_{12}, \label{eqn603}\\
			A_{21}Q_1^*&=Q_2^*A_{21}, \quad A_{22}Q_2^*=Q_2^*A_{22}. \label{eqn604}
		\end{align}
		The condition that $AB=QBA$ implies that 
		\begin{align}
			A_{11}B_1 &=Q_1B_1A_{11}, \quad  A_{12}B_2=Q_1B_1A_{12}, \label{eqn605}  \\
			A_{21}B_1 &=Q_2B_2A_{21},  \quad A_{22}B_2=Q_2B_2A_{22} \label{eqn606}.
		\end{align}	
		Using the hypothesis $AB^*=Q^* B^*A$, we have that
		\begin{align}
			A_{11}B_1^* &=Q_1^* B_1^*A_{11}, \quad  A_{12}B_2^*=Q_1^* B_1^*A_{12}, \label{eqn607}  \\
			A_{21}B_1^* &=Q_2^* B_2^*A_{21},  \quad A_{22}B_2^*=Q_2^* B_2^*A_{22} \label{eqn608}.
		\end{align}	
		We show that $A_{21}=0=A_{12}$. Let us consider the closed linear subspace $\HS_{21}=\overline{A_{21}\HS_1}$ of $\HS_2$. It follows from (\ref{eqn606}) and (\ref{eqn604}) that
		\begin{equation}\label{eqn609}
			B_2A_{21}\HS_1=Q_2^* A_{21}B_1\HS_1=A_{21}Q_1^*B_1\HS_1 \quad \text{and} \quad B_2^*A_{21}\HS_1=Q_2A_{21}B_1^*\HS_1=A_{21}Q_1B_1^*\HS_1 
		\end{equation}
		respectively. Note that last equalities in (\ref{eqn609}) follow from (\ref{eqn604}) and (\ref{eqn602}) respectively. Since $\HS_1$ is a reducing subspace for $B_1$ and $Q_1$, it follows from (\ref{eqn609})  that 
		\[
		B_2A_{21}\HS_1 \subseteq A_{21}\HS_1 \quad \text{and} \quad B_2^*A_{21}\HS_1 \subseteq A_{21}\HS_1.
		\]
		Using continuity arguments, one can show that $B_2\HS_{21} \subseteq \HS_{21}$ and $B_2^*\HS_{21} \subseteq \HS_{21}$. Thus $\HS_{21}$ is a closed reducing subspace for $B_2$. Again using (\ref{eqn609}), for any $x \in \HS_1$, we have
		\begin{equation*}
			\begin{split}
				B_2^*B_2A_{21}x =B_2^* Q_2^*A_{21}B_1x =Q_2^*B_2^*A_{21}B_1x=Q_2^*Q_2A_{21}B_1^*B_1x = A_{21}x, 
			\end{split}
		\end{equation*}
		where in the last equality we used the fact that $B_1$ and $Q_1$ are unitaries on $\HS_1$. Similarly, it follows from (\ref{eqn609}) that
		\begin{equation*}
			\begin{split}
				B_2B_2^*A_{21}x=B_2Q_2A_{21}B_1^*x=Q_2B_2A_{21}B_1^*x=Q_2Q_2^*A_{21}B_1B_1^*x=A_{21}x
			\end{split}
		\end{equation*}
		for every $x \in \HS_1$. Using continuity, we have that $B_2^*B_2y=y=B_2B_2^*y$ for every $y \in \HS_{21}$. Thus, $\HS_{21}$ is a closed subspace of $\HS_2$ which reduces $B_2$ and $B_2|_{\HS_{21}}$ is a unitary. Since $B_2$ is a c.n.u. contraction, we must have $\HS_{21}=\{0\}$ and so, $A_{21}\HS_1=\{0\}$. Consequently, $A_{21}=0$. Now, we have 
		\[
		A=\begin{bmatrix}
			A_{11} & A_{12} \\
			0 & A_{22}		
		\end{bmatrix}
		\quad
		\text{and so}, 
		\quad
		A^*=\begin{bmatrix}
			A_{11}^* & 0 \\
			A_{12}^* & A_{22}^*		
		\end{bmatrix}
		\]
		with respect to $\HS=\HS_1 \oplus \HS_2$. We now prove that $A_{12}=0$ and we employ similar techniques as in the proof of $A_{21}=0$. By doubly $Q$-commutativity condition of $A$ and $B$, it follows that 
		\[
		A^*B^*=QB^*A^* \quad \text{and} \quad A^*B=Q^* BA^*.
		\]
		Using the above operator equations, a simple calculation gives that 
		\begin{equation}\label{eqn610}
			A_{12}^*B_1^*=Q_2B_2^*A_{12}^* \quad \text{and} \quad A_{12}^*B_1=Q_2^* B_2A_{12}^*.
		\end{equation}
		Much similar to the prior case, it follows from (\ref{eqn601}), (\ref{eqn603}), (\ref{eqn610}) and some continuity arguments that $\overline{A_{12}^*\HS_1}$ is a closed subspace of $\HS_2$ that reduces $B_2$ to a unitary. 
		Since $B_2$ is a c.n.u. contraction, it follows that $A_{12}^*\HS_1=\{0\}$ and so, $A_{12}=0$. Therefore, $\HS_1$ and $\HS_2$ are reducing subspaces for $A$ as well. The proof is now complete.
	\end{proof}

The decomposition results for any finite family of doubly $Q$-commuting contractions was proved in \cite{Maji}. We extend those results to any arbitrary family of doubly $Q$-commuting contractions in the light of combinatorial techniques developed in \cite{Pal}. Prior to this, we first establish certain terminologies as in \cite{Pal} to avoid any confusion here. 
	
\begin{defn}
		Let $T$ be a contraction acting on a Hilbert space $\mathcal{H}$.
		\begin{enumerate}
			\item  $T$ is an \textit{atom} if $T$ is either a unitary or a c.n.u. contraction.
			\item  $T$ is an \textit{atom of type $A_u$} if $T$ is a unitary. 
			\item $T$ is an \textit{atom of type $A_{cnu}$} if $T$ is a c.n.u. contraction.
			\item  $T$ is a \textit{non-atom} if $T$ is neither a unitary nor a c.n.u. contraction.
		\end{enumerate}
	\end{defn} 	
	
Let $\mathcal{T}=\{T_r \ : \ r \in R  \}$ be a doubly $Q$-commuting family of contractions acting on a Hilbert space $\HS$, where  $Q=\{Q_{rs} : Q_{rs}=Q_{sr}^*, r \ne s \in R\}$ is a commuting family of unitaries on $\HS$ such that $Q_{st}T_r=T_rQ_{st}$ for all $r, s, t$ in $R$ with $s \ne t$ and $R$ is an arbitrary set. We assume that there are infinitely many non-atoms in $\mathcal{T}$ to avoid triviality. We choose an arbitrary $r_1 \in R$. If $T_{r_1}^{(u)} \oplus T_{r_1}^{(c)}$ is the canonical decomposition of $T_{r_1}$ with respect to $\mathcal{H}=\HS^{(u)} \oplus \HS^{(c)}$ as in Theorem \ref{thm104}, then $\mathcal{H}^{(u)}, \mathcal{H}^{(c)}$ are joint reducing subspaces for $\mathcal{T}$ and $Q$ by Theorem \ref{Qreducing}. Thus, $\mathcal{T}$ is orthogonally decomposed into two tuples given by
	\[
	\mathcal{T}^{(u)}=\{T_s|_{\HS^{(u)}}  \ : s \in R \} \quad \text{and} \quad \mathcal{T}^{(c)}=\{T_s|_{\HS^{(c)}}  \ : s \in R \}.
	\]
	Clearly, the $r_1$-th entities ($T_{r_1}^{(u)}$ and $T_{r_1}^{(c)}$, respectively) are atoms. Let 
	\[
	\Delta_0=\{(\mathcal{T}, \HS)\} \quad \text{and} \quad \Delta_{r_1}=\{(\mathcal{T}^{(u)}, \HS^{(u)} ), (\mathcal{T}^{(c)}, \HS^{(c)} )\}.
	\]
	Next, we apply the axiom of choice and choose $r_2 \in R \setminus \{r_1\}$. Consider the $r_2$-th components of $\mathcal{T}^{(u)}$ and $\mathcal{T}^{(c)}$, i.e. $T_{r_2}|_{\HS^{(u)} }$ and $T_{r_2}|_{\HS^{(c)}}$ to perform the canonical decomposition of $\mathcal{T}^{(u)}$ and $\mathcal{T}^{(c)}$ respectively. We obtain $2^2=4$ new tuples acting on four orthogonal parts of $\HS$, say $\HS^{(uu)}, \HS^{(uc)}, \HS^{(cu)} , \HS^{(cc)}$, where $\HS^{(u)} =\HS^{(uu)}  \oplus \HS^{(uc)} $ and $\HS^{(c)} =\HS^{(cu)}  \oplus \HS^{(cc)} $. Let us denote by 
	\[
	\mathcal{T}^{(uu)}=\{T_s|_{\HS^{(uu)} }  \ : s \in R \} \quad \text{and} \quad \mathcal{T}^{(uc)}=\{T_s|_{\HS^{(uc)}}  \ : s \in R \}
	\]
	and 
	\[
	\mathcal{T}^{(cu)}=\{T_s|_{\HS^{(cu)} }  \ : s \in R \} \quad \text{and} \quad \mathcal{T}^{(cc)}=\{T_s|_{\HS^{(cc)}}  \ : s \in R \}
	\]
	respectively. Let
	\[
	\Delta_{r_2}=\left\{(\mathcal{T}^{(uu)}, \HS^{(uu)}), \ (\mathcal{T}^{(uc)}, \HS^{(uc)}), \ (\mathcal{T}^{(cu)}, \HS^{(cu)}), \ (\mathcal{T}^{(cc)}, \HS^{(cc)})\right\}.
	\]
	One can show that $\HS^{(uu)} , \HS^{(uc)} , \HS^{(cu)} $ and $\HS^{(cc)}$ are closed reducing subspaces for $\mathcal{T}$ and $Q$ such that
	\begin{equation}\label{eqn307}
		\left.\begin{split}
			& (a) \quad \text{$T_{r_1}|_{\HS^{(uu)} }$ and $T_{r_2}|_{\HS^{(uu)} }$ are unitaries};\\
			& (b) \quad \text{$T_{r_1}|_{\HS^{(uc)} }$ is a unitary and $T_{r_2}|_{\HS^{(uc)} }$ is a c.n.u. contraction};\\
			& (c) \quad \text{$T_{r_1}|_{\HS^{(cu)} }$ is a c.n.u. contraction and $T_{r_2}|_{\HS^{(cu)} }$ is a unitary};\\
			& (d) \quad \text{$T_{r_1}|_{\HS^{(cc)} }$ and $T_{r_2}|_{\HS^{(cc)} }$ are c.n.u. contractions}.\\
		\end{split}\right\}
	\end{equation}
	Let if possible $\mathcal{L}^{(uu)}$ be a closed reducing subspace for $T_{r_1}$ and $T_{r_2}$ such that $T_{r_1}|_{\mathcal{L}^{(uu)}}$ and $T_{r_2}|_{\mathcal{L}^{(uu)}}$ are unitaries.  It follows from Theorem \ref{thm104} that $\HS^{(u)} $ is the largest closed subspace of $\HS$ that reduces $T_{r_1}$ to a unitary and so, $\mathcal{L}^{(uu)} \subseteq \HS^{(u)} $. Again by Theorem \ref{thm104}, we get that $\HS^{(uu)} $ is the largest among all subspaces of $\HS^{(u)} $ which reduces $T_{r_2}$ to a unitary. Consequently, $\mathcal{L}^{(uu)} \subseteq \HS^{(uu)} $. Similary, one can prove that $\HS^{(uc)}, \HS^{(cu)}$ and $\HS^{(cc)}$ are maximal subspaces of $\HS$ in the sense of (\ref{eqn307}). After applying the axiom of choice $n$-times, we obtain $2^n$ operator tuples (corresponding to the $2^n$ orthogonal splits of $\HS$) extending the decomposition for a pair of doubly $Q$-commuting contractions as in (\ref{eqn307}). En route, we also obtain the families $\{\Delta_0, \Delta_{r_k} : 1 \leq k \leq n \}$. Now, we have the following decomposition theorem when $R$ is a finite family of doubly $Q$-commuting contractions.
		
	\begin{thm}\label{thm303F} 
		Let $(T_{1}, \ldots, T_{n})$ be a doubly $Q$-commuting tuple of contractions acting on a Hilbert space $\HS$. Then there corresponds a decomposition of $\mathcal{H}$ into an orthogonal sum of $2^{n}$ joint reducing subspaces $\mathcal{H}_1, \ldots, \mathcal{H}_{2^{n}}$ for $T_{1}, \dotsc, T_{n}$ and $Q$ such that the following hold.
		\begin{enumerate}
			\item  $T_{i}|_{\mathcal{H}_j}$ is either a unitary or a c.n.u. contraction for $1 \leq i \leq n$ and  $ 1 \leq j \leq 2^{n}$;
			
			\item  $\mathcal{H}_1$ is the largest closed joint reducing subspace of $\HS$ on which each $T_{j}$ acts as a unitary;
			
			\item  $\mathcal{H}_{2^{n}}$ is the largest closed joint reducing subspace of $\HS$ on which each $T_{j}$ is a c.n.u. contraction;
			
			\item If $\Theta_n$ is the set of all functions $\theta: \{1, \dotsc, n\} \to \{A_u, A_{cnu}\}$, then for every $\theta \in \Theta_n$, there corresponds a unique subspace $\HS_{t(\theta)} (1 \leq t(\theta) \leq 2^n)$ such that the $k$-th component of $(T_1|_{\HS_{t(\theta)}}, \dotsc, T_n|_{\HS_{t(\theta)}})$ is of the type $\theta(k)$ for $1 \leq k \leq n$ and $\HS_{t(\theta)}$ is the largest such joint reducing subspace for $T_1, \dotsc, T_n$ and $Q$.
		\end{enumerate}	
	\end{thm} 
	
	\begin{proof}
		
Follows from the above discussion with an application of mathematical induction on $n$. However, a proof to this could be found 	in \cite{Maji}.
	\end{proof}
	
\subsection{The countable infinite case.} 
In this Subsection, we imitate the definitions, terminologies and a few results from the first named author's work \cite{Pal}. We follow the notations in the discussion above Theorem \ref{thm303F}. Let us assume that $R$ is a countable set and so, $R=\{r_n \ : \ n \in \N\}$. One can follow the natural order as $1 \to 2 \to 3 \to \dotsc$ which induces an order $r_1 \to r_2 \to r_3 \to \ldots$ on $R$. Let 
	\[
	\Delta_\N=\Delta_0 \cup \underset{n \in \N}{\bigcup}\Delta_{r_n}.
	\]
	For any two members $\left(\mathcal{T}_\alpha, \mathcal{H}_{\alpha}\right),\left(\mathcal{T}_\beta, \mathcal{H}_{\beta}\right)$ in $\Delta_\N$, we say that  $\left(\mathcal{T}_\alpha, \mathcal{H}_{\alpha}\right) \leq \left(\mathcal{T}_\beta, \mathcal{H}_{\beta}\right)$ if and only if $\mathcal{H}_{\beta} \subseteq \mathcal{H}_{\alpha}, \mathcal{H}_{\beta}$ is a joint reducing subspace for $\mathcal{T}_\alpha, Q$ and $\mathcal{T}_\beta=\mathcal{T}_\alpha|_{{\mathcal{H}_{\beta}}}$. Thus,  we obtain a partially ordered set $(\Delta_\N, \leq)$. Evidently, $(\mathcal{T}, \HS) \leq (\mathcal{T}_\alpha, \HS_{\alpha})$ for any $(\mathcal{T}_\alpha, \HS_{\alpha}) \in \Delta_\N$.
	
	\begin{defn}
		A proper subset $\Delta_*$ of $\Delta_{\mathbb{N}}$ is called a \textit{maximal totally ordered set} if $(\Delta_*, \leq)$ is a totally ordered subset of $\left(\Delta_\mathbb{N}, \leq \right)$ and for any $\left(\mathcal{T}_\alpha, \mathcal{H}_{\alpha}\right) \in \Delta_{\mathbb{N}} \backslash \Delta_*$, the set
		$
		\Delta_* \cup \{(\mathcal{T}_\alpha, \mathcal{H}_{\alpha})\}
		$
		is not a totally ordered subset of $\Delta_{\mathbb{N}}$.
\end{defn}

A maximal totally ordered set can be written as an increasing chain of elements in $\Delta_\N$. Note that only one element from each $\Delta_{r_n}$ appears in a maximal totally ordered set for $n \in \N$. Using the above notations, one can verify that the following are examples of maximal totally ordered sets in $\Delta_\N$ :  
\[
\left\{(\mathcal{T}, \HS), (\mathcal{T}^{(u)}, \HS^{(u)}), (\mathcal{T}^{(uu)}, \HS^{(uu)}), \dotsc \right\}, \, \,
\left\{(\mathcal{T}, \HS), (\mathcal{T}^{(c)}, \HS^{(c)}), (\mathcal{T}^{(cu)}, \HS^{(cu)}), \dotsc \right\}.
\]	
Given two maximal totally ordered sets $\Delta_{*}^{(1)}$ and $\Delta_{*}^{(2)}$ in $\Delta_\N$ in the form of increasing chains
\[
\Delta_{*}^{(1)}=\left\{(\mathcal{T}_{\alpha_1}^{(1)}, \HS_{\alpha_1}^{(1)}), \ (\mathcal{T}_{\alpha_2}^{(1)}, \HS_{\alpha_2}^{(1)}), \ (\mathcal{T}_{\alpha_3}^{(1)}, \HS_{\alpha_3}^{(1)}), \dotsc  \right\}
 \]
 \[
 \Delta_{*}^{(2)}=\left\{(\mathcal{T}_{\alpha_1}^{(2)}, \HS_{\alpha_1}^{(2)}), \ (\mathcal{T}_{\alpha_2}^{(2)}, \HS_{\alpha_2}^{(2)}), \ (\mathcal{T}_{\alpha_3}^{(2)}, \HS_{\alpha_3}^{(2)}), \dotsc  \right\},
\]
we say that $\Delta_{*}^{(1)}$ and $\Delta_{*}^{(2)}$ are \textit{different} if there is some $m \in \N$ such that 
$
(\mathcal{T}_{\alpha_m}^{(1)}, \HS_{\alpha_m}^{(1)}) \ne (\mathcal{T}_{\alpha_m}^{(2)}, \HS_{\alpha_m}^{(2)})$. Thus, if $\Delta_{*}^{(1)}$ and $\Delta_{*}^{(2)}$ are different, then it follows that $(\mathcal{T}_{\alpha_k}^{(1)}, \HS_{\alpha_k}^{(1)}) \ne (\mathcal{T}_{\alpha_k}^{(2)}, \HS_{\alpha_k}^{(2)})$ for all $k \geq m$.

\begin{defn}
	Let $\Delta_{*}=\left\{(\mathcal{T}_{\alpha_1}, \HS_{\alpha_1}), \ (\mathcal{T}_{\alpha_2}, \HS_{\alpha_2}), \ (\mathcal{T}_{\alpha_3}, \HS_{\alpha_3}), \dotsc  \right\}$ be a maximal totally ordered set with elements written in increasing order. With the notations 
	$
	\HS_*=\overset{\infty}{\underset{n=1}{\bigcap}}\HS_{\alpha_n}$ and $\mathcal{T}_*=\mathcal{T}|_{\HS_*}$, the pair $(\mathcal{T}_*, \HS_*)$ is called the \textit{maximal element} for $\Delta_*$.
	\end{defn}
	
We recall from \cite{Pal} a few results about the maximal element of a maximal totally ordered set.

\begin{lem}[\cite{Pal}, Lemmas 6.5 \& 6.6]\label{lem_max}
	The operator tuple in a maximal element consists of atoms only. The maximal elements of two different maximal totally ordered sets are different.
\end{lem} 	
	
Let $\Delta_\infty$ be the set of all maximal elements for the maximal totally ordered sets in $\Delta_\N$. It follows from Lemma \ref{lem_max} that any element in $\Delta_\infty$ contains atoms only. 

\begin{lem}[\cite{Pal}, Lemmas 6.7 \& 6.9]\label{lem_dinfty}
For every sequence $\alpha: \N \to \{A_u, A_{cnu}\}$, there exists a unique maximal element $(\mathcal{T}_\alpha, \HS_\alpha)$ in $\Delta_\infty$ such that the $k$-th component of $\mathcal{T}_\alpha$ is of type $\alpha(k)$ for every $k \in \N$. Moreover, the cardinality of $\Delta_\infty$ is $2^{\aleph_0}$, where $\aleph_0$ is the cardinality of $\N$.	
\end{lem}

\noindent \textbf{Notation.} Let $S(\N)$ be the collection of all sequences $\alpha: \mathbb{N} \rightarrow\{A_u, A_{cnu}\}$. We represent $\Delta_\infty$ as
\[
\Delta_\infty=\left\{ (\mathcal{T}_\alpha, \HS_\alpha) : \alpha \in S(\N) \right\}.
\] 
Following the proof of Lemmas 6.11 and 6.12 in \cite{Pal}, we arrive at the subsequent result.

\begin{lem}[\cite{Pal}, Lemmas 6.11 \& 6.12]\label{lem_final}
For the operator tuple $\mathcal{T}$ on $\HS$, we have that
	\[
	\mathcal{H}=\bigoplus_{\alpha \in S(\N)} \mathcal{H}_\alpha \quad \text{and} \quad 	\mathcal{T}=\bigoplus_{\alpha \in S(\N)} \mathcal{T}_\alpha.
	\]
Given $n \in \N$ and $(\mathcal{T}_\alpha, \HS_\alpha) \in \Delta_{r_n}$, there are infinitely many maximal totally ordered sets $\Delta_*$ in $\Delta_\N$ such that $(\mathcal{T}_\alpha, \HS_\alpha) \in \Delta_*$.	
\end{lem}

All the above mentioned results from \cite{Pal} are set-theoretic in nature and hence, holds in the setting of doubly $Q$-commuting contractions too. By combining Lemmas \ref{lem_max}, \ref{lem_dinfty} and \ref{lem_final}, we arrive at the following generalization of Theorem \ref{thm303F} in the countably infinite setting.

	\begin{thm}\label{thm_count}
		Let $\mathcal{T}=\left(T_{n}\right)_{n=1}^{\infty}$ be a doubly $Q$-commuting family of contractions acting on a Hilbert space $\mathcal{H}$ which consists of infinitely many non-atoms and let $S(\N)$ be the set of all sequences $\alpha: \mathbb{N} \rightarrow\{A_u, A_{cnu}\}$. Then $\mathcal{H}$ admits an orthogonal decomposition $\mathcal{H}=\bigoplus_{\alpha \in S(\N)} \mathcal{H}_\alpha$ such that each $\mathcal{H}_\alpha$ is a joint reducing subspace for $\mathcal{T}, Q$ and the following assertions hold.
		\begin{enumerate}
			
\item  For each $\alpha \in S(\N)$, the components of the operator tuple $\left(\left.T_{n}\right|_{\mathcal{H}_\alpha}\right)_{n=1}^{\infty}$ are all atoms.

\item  The cardinality of $\Delta_{\infty}=\left(\mathcal{T}_\alpha, \mathcal{H}_\alpha\right)_{\alpha \in S(\N)}$ is $2^{\aleph_0}$, where $\mathcal{T}_\alpha=\left(T_{n}|_{\mathcal{H}_\alpha} \right)_{n=1}^{\infty}$ for $\alpha \in S(\N)$.

\item  There is a unique element, say $\left(\mathcal{T}_1, \mathcal{H}_1\right)$ in $\Delta_{\infty}$ such that the components of $\mathcal{T}_1$ are all atoms of type $A_u$ which is determined by the constant sequence $\alpha_1: \mathbb{N} \rightarrow\{A_u, A_{cnu} \}$ defined by $\alpha_1(n)=A_u$ for all $n \in \N$.

\item  There is a unique element, say $\left(\mathcal{T}_{2^{\aleph_0}}, \mathcal{H}_{2^{\aleph_0}}\right)$ in $\Delta_{\infty}$ such that the components of $\mathcal{T}_{2^{\aleph_0}}$ are all atoms of type $A_{cnu}$, and this is determined by the constant sequence $\alpha_{2^{\aleph_0}}: \mathbb{N} \rightarrow\{A_u, A_{cnu}\}$ defined by $\alpha_{2^{\aleph_0}}(n)=A_{cnu}$ for all $n \in \mathbb{N}$.
\end{enumerate}
One or more members of $\left\{\mathcal{H}_\alpha: \alpha \in S(\N)\right\}$ may coincide with the trivial subspace $\{0\}$.
	\end{thm}

	
	\subsection{The general case.} Let $\mathcal{T}=\{T_r: r \in R \}$ be a doubly $Q$-commuting family of  contractions acting on a Hilbert space $\mathcal{H}$. Here $R$ is any arbitrary set. To avoid triviality, we again assume that $\mathcal{T}$ consists of infinitely many non-atoms. As in the previous Subsection, we obtain the families 
	\[
\{(\mathcal{T}, \HS)\}
\]
\[
\{(\mathcal{T}^{(u)}, \HS^{(u)} ), (\mathcal{T}^{(c)}, \HS^{(c)} )\}
\]
\[ \{(\mathcal{T}^{(uu)}, \HS^{(uu)}), \ (\mathcal{T}^{(uc)}, \HS^{(uc)}), \ (\mathcal{T}^{(cu)}, \HS^{(cu)}), \ (\mathcal{T}^{(cc)}, \HS^{(cc)})\}.
	\]
Using axiom of choice and Theorem \ref{thm104}, we continue the orthogonal decompositions this way. Finally, we obtain the limiting set $\Delta_\infty$ which must have cardinality $2^{|R|}$, where $|R|$ is the cardinality of $R$. This holds, because here we consider the set of all functions $\phi: R \to \{A_u, A_{cnu}\}$ which has cardinality $2^{|R|}$. One can define ``$\leq$" similarly as in the countable infinite case and establish analogues of Lemmas \ref{lem_max}, \ref{lem_dinfty} and \ref{lem_final}. Consequently, we have the following analogue of Theorems \ref{thm303F} and \ref{thm_count} for an arbitrary doubly $Q$-commuting family of contractions. This is the main result of this Section.	
	\begin{thm}\label{thm303}
		Let $\mathcal{T}=\{T_r: r \in R \}$ be a doubly $Q$-commuting family of  contractions acting on a Hilbert space $\mathcal{H}$ that consists of infinitely many non-atoms. If $S(R)$ is the set of all functions $\phi: R \rightarrow\{A_u, A_{cnu}\}$ and if
		$
		\Delta_\infty=\{(\mathcal{T}_\phi, \mathcal{H}_\phi): \phi\in S(R)\}
		$
		with $\mathcal{T}_\phi=\{T_r|_{\mathcal{H}_\phi}: r \in R\}$, then the following statements hold.
		
		\begin{enumerate}
			\item  $\mathcal{H}$ admits an orthogonal decomposition $\mathcal{H}=\bigoplus_{\phi \in S(R)} \mathcal{H}_\phi$ such that each $\mathcal{H}_\phi$ is a joint reducing subspace for $\mathcal{T}$ and $Q$.
			
\item  For each $\phi \in S(R)$, the members of the operator family $\{T_r|_{\mathcal{H}_\phi}: r \in R \}$ are all atoms.

\item  There is exactly one element, say $\left(\mathcal{T}_1, \mathcal{H}_1\right)$ in $\Delta_{\infty}$ such that the members of $\mathcal{T}_1$ are all atoms of type $A_u$ and this is determined by the constant function $\phi_1: R \rightarrow\{A_u, A_{cnu}\}$ defined by $\phi_1(r)=A_u$ for all $r \in R$.

			\item  There is exactly one element, say $\left(\mathcal{T}_{2^{|R|}}, \mathcal{H}_{2^{|R|}}\right)$ in $\Delta_\infty$ such that the members of $\mathcal{T}_{2^{|R|}}$ are all atoms of type $A_{cnu}$, and this is determined by the constant function $\phi_{2^{|R|}}: R \rightarrow\{A_u, A_{cnu}\}$ defined by $\phi_{2^{|R|}}(r)=A_{cnu}$ for all $r \in R$.
		
			\item  The cardinality of the set $\Delta_{\infty}$ is $2^{|R|}$, where $|R|$ is the cardinality of the set $R$.
			
		\end{enumerate}
		One or more members of $\left\{\mathcal{H}_\phi: \phi \in S(R)\right\}$ may coincide with the trivial subspace $\{0\}$.
	\end{thm}


	\subsection{Wold decomposition of doubly $Q$-commuting isometries.} We obtain the Wold-type decomposition result for doubly $Q$-commuting isometries as a special case of Theorem \ref{thm303}. To explain it more clearly, we need the following definition. 
	\begin{defn}
		A Hilbert space contraction $V$ is said be of \textit{type $A_{shift}$} if it is a unilateral shift. 
	\end{defn}
Note that if $V$ is an isometry which is a c.n.u. contraction as well, then $V$ is a pure isometry and hence a unilateral shift. Thus, every isometry of type $A_{cnu}$ is of type $A_{shift}$. We now present the following result whose proof is a consequence of Theorem \ref{thm303}.	
	
\begin{thm}\label{Wold general case}
		Let $\mathcal{V}=\{V_r: r \in R \}$ be a doubly $Q$-commuting family of isometries acting on a Hilbert space $\mathcal{H}$ which consists of infinitely many isometries that are neither unitaries nor pure isometries. If $V(R)$ is the set of all functions $\mu: R \rightarrow\{A_u, A_{shift}\}$ and 
		\[
		\Delta_\infty=\{(\mathcal{V}_\mu, \mathcal{H}_\mu): \mu \in V(R)\},
		\]
		where $\mathcal{V}_\mu=\{V_r|_{\mathcal{H}_\mu}: r \in R\},$ then we have the following assertions.
		\begin{enumerate}
			\item  $\mathcal{H}$ admits an orthogonal decomposition $\mathcal{H}=\bigoplus_{\mu \in V(R)} \mathcal{H}_\mu$ such that each $\mathcal{H}_\mu$ is a joint reducing subspace for $\mathcal{V}$ and $Q$.

	\item  For each $\mu \in V(R)$, the members of the operator family $\{V_r|_{\mathcal{H}_\mu}: r \in R \}$ are either unitaries or pure isometries.
			
			\item  There is exactly one element, say $\left(\mathcal{V}_1, \mathcal{H}_1\right)$ in $\Delta_{\infty}$ such that the members of $\mathcal{V}_1$ are all unitaries, and this is determined by the constant function $\mu_1: R \rightarrow\{A_u, A_{shift}\}$ defined by $\mu_1(r)=A_u$ for all $r \in R$.
	
\item  There is exactly one element, say $\left(\mathcal{V}_{2^{|R|}}, \mathcal{H}_{2^{|R|}}\right)$ in $\Delta_{\infty}$ such that the members of $\mathcal{V}_{2^{|R|}}$ are all pure isometries, and this is determined by the constant function $\mu_{2^{|R|}}: R \rightarrow\{A_u, A_{shift}\}$ defined by $\mu_{2^{|R|}}(r)=A_{shift}$ for all $r \in R$.
	
			\item  The cardinality of the set $\Delta_{\infty}$ is $2^{|R|}$, where $|R|$ is the cardinality of the set $R$.
			
		\end{enumerate}
		One or more members of $\left\{\mathcal{H}_\mu: \mu \in V(R)\right\}$ may coincide with the trivial subspace $\{0\}$.
	\end{thm}
	
	\vspace{0.05cm}

\section{Decomposition of any family of doubly $Q$-commuting c.n.u. contractions}\label{sec04}
	
		\vspace{0.2cm}
	
	\noindent We learn from Theorem \ref{thm303F} that given a tuple of doubly $Q$-commuting contractions $(T_1, \dotsc, T_n)$ acting on a Hilbert space $\HS$, there is a subspace $\HS_{2^n}$ of $\HS$ that is joint reducing for $T_1, \dotsc, T_n$ and $Q$ such that each $T_i|_{\HS_{2^n}}$ is a c.n.u. contraction. In this Section, we further decompose an $n$-tuple of doubly $Q$-commuting c.n.u. contractions into $2^n$ orthogonal parts. Then, we extend this result to any family of doubly $Q$-commuting c.n.u. contractions. The following result plays a central role here.
	
	\begin{thm}\label{Qreducing cnu}
		Let $(A, B)$ be a doubly $Q$-commuting pair of c.n.u. contractions on a Hilbert space $\mathcal{H}$. If $B=B_1\oplus B_2$ is the Levan-type decomposition of $B$  with respect to the orthogonal decomposition $\mathcal{H}=\mathcal{H}_1 \oplus \mathcal{H}_2$ as in Theorem \ref{thm103}, then $\mathcal{H}_1, \mathcal{H}_2$ are reducing subspaces for $A$ and $Q$.
	\end{thm}
	
\begin{proof}
It suffices to show that $\mathcal{H}_1$ reduces $A$ and $Q$. First, we show that $\HS_1$ is invariant under both $Q$ and $Q^*$. Since $Q$ commutes with $A$ and $B$, it follows from Fuglede's theorem \cite{Fuglede} that $Q$ doubly commutes with both $A$ and $B$.  Following the proof of Theorem \ref{thm103} from \cite{Levan} (which is Theorem 1 in \cite{Levan}), we have that $\HS_1$ is the span closure of all closed subspaces of $\HS$ that reduces $B$ to an isometry. Let $\mathcal{L}=\overline{Q\HS_1}$ and $\mathcal{L}_*=\overline{Q^*\HS_1}$. Since $\HS_1$ is a reducing subspace for $B$, we have 
		\[
		BQ\HS_1=QB\HS_1 \subseteq Q\HS_1 \quad \text{and} \quad B^*Q\HS_1=QB^*\HS_1 \subseteq Q\HS_1. 
		\] 
		Using continuity, one can show that $\mathcal{L}$ is a reducing subspace for $B$. For any $x \in \HS_1$, we have $B^*Bx=x$ and so, $B^*BQx=QB^*Bx=Qx$. Again by continuity arguments, it follows that $\mathcal{L}$ is a closed subspace of $\HS$ which reduces $B$ to an isometry and thus, $\mathcal{L} \subseteq \HS_1$. Therefore, $Q\HS_1 \subseteq \HS_1$. Using the same arguments for $\mathcal{L}_*$, it is easy to see that $\mathcal{L}_* \subseteq \HS_1$ and so, $Q^*\HS_1 \subseteq \HS_1$. Consequently, $\HS_1$ is a reducing subspace for $Q$. 
		
		\smallskip
		
	Next, we prove that $\HS_1$ reduces $A$. To see this, first note that $BA=Q^*AB$ and $QAB^*=B^*A$ since $A$ and $B$ doubly $Q$-commute. Let $x \in \HS_1$. Since $\HS_1$ is a joint reducing subspace for $B, Q$ and $(A, Q)$ is a doubly commuting pair, we have that
		\begin{equation*}
			BAx=Q^*ABx=AQ^*Bx \in A\mathcal{H}_1,  \quad  B^*Ax=QAB^*x=AQB^*x \in A\mathcal{H}_1,
		\end{equation*}	
		and that
		\begin{equation*}
			B^*BAx=B^*Q^*ABx=Q^*B^*ABx=Q^*QAB^*Bx=Ax.
		\end{equation*}	
The last equality in the above equation follows from the fact that $B|_{\mathcal{H}_1}$ is an isometry. Using continuity argument, we have that $\mathcal{H}_0=\overline{A\mathcal{H}_1}$ is a closed subspace of $\mathcal{H}$ which reduces $B$ such that $B|_{\mathcal{H}_0}$ is isometry. Therefore, it follows from the definition of $\mathcal{H}_1$ that $A\mathcal{H}_1 \subseteq  \mathcal{H}_1$. Similarly, one can prove that $A^*\mathcal{H}_1 \subseteq \mathcal{H}_1$ and this completes the proof.  
	\end{proof}
	
We now prove a Levan-type decomposition result for any doubly $Q$-commuting tuple of c.n.u. contractions. For the ease of computations, we use the following terminologies that were introduced in \cite{Pal}.
	
	\begin{defn}
		Let $T$ be a contraction acting on a Hilbert space $\mathcal{H}$.
		\begin{enumerate}
			\item[(i)]  $T$ is a \textit{fundamental c.n.u. contraction} if $T$ is either a pure isometry or a c.n.i. contraction.
			\item[(ii)]  $T$ is a \textit{c.n.u. contraction of type $A_{shift}$} if $T$ is a pure isometry. 
			\item[(iii)] $T$ is a \textit{c.n.u. contraction of type $A_{cni}$} if $T$ is a c.n.i. contraction.
			\item[(iv)]  $T$ is a \textit{non-fundamental c.n.u. contraction} if $T$ is a c.n.u. contraction which is neither a pure isometry nor a c.n.i. contraction.
		\end{enumerate}
	\end{defn} 
	
\begin{thm}\label{Levan tuple}
		Let $(T_1, \ldots, T_n)$ be a doubly $Q$-commuting tuple of c.n.u. contractions acting on a Hilbert space $\mathcal{H}$. Then there exists a decomposition of $\mathcal{H}$ into an orthogonal sum of $2^n$ joint reducing subspaces $\mathcal{H}_1, \ldots, \mathcal{H}_{2^n}$ for $T_1, \dotsc, T_n, Q$ such that the following hold.
		\begin{enumerate}
			\item  $T_i|_{\mathcal{H}_j}$ is a fundamental c.n.u. contraction for  $1 \leq i\leq  n$ and $1\leq j \leq 2^n$.

			\item  $\mathcal{H}_1$ is the largest closed subspace of $\HS$ reducing each $T_i$ to a pure isometry.

\item  $\mathcal{H}_{2^n}$ is the largest closed subspace of $\HS$ reducing each $T_i$ to a c.n.i. contraction.
			
			\item If $G_n$ is the set of all functions $g: \{1, \dotsc, n\} \to \{A_{shift}, A_{cni}\}$, then for every $g \in G_n$, there corresponds a unique subspace $\HS_{t(g)} (1 \leq t(g) \leq 2^n)$ such that the $k$-th component of $(T_1|_{\HS_{t(g)}}, \dotsc, T_n|_{\HS_{t(g)}})$ is of the type $g(k)$ for $1 \leq k \leq n$ and $\HS_{t(g)}$ is the largest such joint reducing subspace for $T_1, \dotsc, T_n, Q$.
			
		\end{enumerate}	
		One or more members of $\{\mathcal{H}_j: 1 \leq j \leq 2^n \}$ may coincide with the trivial subspace $\{0\}$.
	\end{thm}
	
		\begin{proof}
We start with a tuple of c.n.u. contractions $(T_1, \dotsc, T_n)$ and a family of commuting unitaries $Q=\{Q_{ij} \ : \ 1 \leq i < j \leq n\}$ on $\HS$  such that 
		\[
		T_iT_j=Q_{ij}T_jT_i, \quad T_iT_j^*=Q_{ij}^*T_j^*T_i, \quad Q_{ji}=Q_{ij}^* \quad  \text{and} \quad T_kQ_{ij}=Q_{ij}T_k
		\]
		for all $1 \leq i < j \leq n$ and $1 \leq k \leq n$. Let $\mathcal{H}=\HS_p \oplus \HS_c$ be the Levan-type decomposition for $T_1$ as in Theorem \ref{thm103}. By Fuglede's theorem \cite{Fuglede}, each $Q_{ij}$ doubly commutes with $T_1$.  It follows from Theorem \ref{Qreducing cnu} that $\mathcal{H}_u, \mathcal{H}_c$ are joint reducing subspaces for $T_2, \dotsc, T_n$ and $Q$ as well. By performing the Levan-type decomposition of $T_{2}|_{\mathcal{H}_u}$ and $T_{2}|_{\mathcal{H}_c}$, we get $2^2$ orthogonal splits of $\HS$. Indeed, a repeated application of Theorems  \ref{thm103} \& \ref{Qreducing cnu} provides orthogonal decompositions $\HS_p=\HS_{pp} \oplus \HS_{pc}$ and $\HS_c=\HS_{cp} \oplus \HS_{cc}$ such that $\HS_{pp}, \HS_{pc}, \HS_{cp}, \HS_{cc}$ are closed reducing subspaces for $T_1, \dotsc T_n$ and $Q$. Also, $T_{1}|_{\HS_{pp}}, T_{1}|_{\HS_{pc}}, T_{2}|_{\HS_{pp}}, T_{2}|_{\HS_{cp}}$ are pure isometries and $T_{1}|_{\HS_{cp}}, T_{1}|_{\HS_{cc}}, T_{2}|_{\HS_{pc}}, T_{2}|_{\HS_{cc}}$ are c.n.i. contractions. Again using the double commutativity of each $Q_{ij}$ with $T_2$, it follows from Theorem \ref{Qreducing cnu} that $\HS_u$ and $\HS_c$ are joint reducing subspaces for $Q$. We move one step ahead and perform the canonical decomposition on $T_3|_{\HS_{pp}}, T_3|_{\HS_{pc}}, T_3|_{\HS_{cp}} T_3|_{\HS_{cc}}$ to obtain the orthogonal decompositions 
		\[
		\HS_{pp}=\HS_{ppp}\oplus \HS_{ppc}, \quad  \HS_{pc}=\HS_{pcp}\oplus \HS_{pcc}, \quad  \HS_{cp}=\HS_{cpp}\oplus \HS_{cpc}, \quad  \HS_{cc}=\HS_{ccp}\oplus \HS_{ccc}
		\] 
		such that each of the spaces in the above decomposition reduce $T_1, \dotsc , T_n$ and $Q$ by virtue of Theorem \ref{Qreducing cnu}. Moreover, we have 
		\[ 
		T_k|_{\mathcal{H}_{i_1i_2i_3}}= \left\{
		\begin{array}{ll}
			\mbox{pure isometry}, & i_k=p \\
			\mbox{c.n.i. contraction}, & i_k=c\\
		\end{array} 
		\right. 
		\]
		for $1 \leq k \leq 3$. Continuing this way, after finitely many steps we obtain a decomposition of $\mathcal{H}$ into an orthogonal sum of $2^{n}$ joint reducing subspaces $\mathcal{H}_1, \ldots, \mathcal{H}_{2^{n}}$ of $T_1, \dotsc, T_{n}, Q$ such that the desired conclusion follows.
	\end{proof}

Following Example 3.5 in \cite{Maji}, we present examples of doubly $q$-commuting pairs of c.n.u. and c.n.i. contractions.

\begin{eg}\label{eg404}
For a Hilbert space $\mathcal{E}$, the vector-valued Hardy space of the bidisk $\mathbb D^2$ is defined as
\[
H^2_{\mathcal{E}}(\D^2)=\left\{f(z_1, z_2)=\overset{\infty}{\underset{i, j=0}{\sum}} f_{ij} \ z_1^iz_2^j \ : \ z_1, z_2 \in \D, f_{ij} \in \mathcal{E}, \overset{\infty}{\underset{i, j=0}{\sum}}\|f_{ij}\|_{\mathcal{E}}^2 < \infty \right\}
\]
with inner product
$
\langle f(z_1, z_2), g(z_1, z_2) \rangle =\overset{\infty}{\underset{i, j=0}{\sum}}\langle f_{ij},{g}_{ij}\rangle_{\mathcal{E}}.
$
Let $q$ be a unimodular scalar. Define $T_1, T_2: H^2_{\mathcal{E}}(\D^2) \to H^2_{\mathcal{E}}(\D^2)$ by
$
T_1(f(z_1, z_2))=z_2f(qz_1, z_2)$ and $T_2(f(z_1, z_2))=z_1f(z_1, z_2)$.
The adjoints of $T_1$ and $T_2$ are given by
\[
T_1^*(f(z_1, z_2))=\frac{f(\overline{q}z_1, z_2)-f(\overline{q}z_1, 0)}{z_2} \quad \text{and} \quad  T_2^*(f(z_1, z_2))=\frac{f(z_1, z_2)-f(0, z_2)}{z_1},
\]
respectively. In the series representation form, we have that
\[
T_1\left(\overset{\infty}{\underset{i, j=0}{\sum}} f_{ij} \ z_1^iz_2^j\right)=\overset{\infty}{\underset{i, j=0}{\sum}} f_{ij} \ (qz_1)^iz_2^{j+1} \quad \text{and} \quad T_2\left(\overset{\infty}{\underset{i, j=0}{\sum}} f_{ij} \ z_1^iz_2^j\right)=\overset{\infty}{\underset{i, j=0}{\sum}} f_{ij} \ z_1^{i+1}z_2^j
\]
with their adjoints given as
\[
T_1^*\left(\overset{\infty}{\underset{i, j=0}{\sum}} f_{ij} \ z_1^iz_2^j\right)=\overset{\infty}{\underset{i=0, j=1}{\sum}} f_{ij} \ (\overline{q}z_1)^iz_2^{j-1} \quad \text{and} \quad T_2^*\left(\overset{\infty}{\underset{i, j=0}{\sum}} f_{ij} \ z_1^iz_2^j\right)=\overset{\infty}{\underset{i=1, j=0}{\sum}} f_{ij} \ z_1^{i-1}z_2^j.
\]
It is easy to see that $(T_1, T_2)$ is a doubly $q$-commuting pair of isometries. For $k \in \mathbb{N}$, note that
\[
T_1^k(H^2_{\mathcal{E}}(\D^2))=\left\{\overset{\infty}{\underset{i, j=0}{\sum}} f_{ij} \ z_1^iz_2^j \in H^2_{\mathcal{E}}(\D^2) : f_{ij}=0 \ \ \text{for} \ \ 0 \leq j \leq k-1 \right\}
\]
and 
\[
T_2^k(H^2_{\mathcal{E}}(\D^2))=\left\{\overset{\infty}{\underset{i, j=0}{\sum}} f_{ij} \ z_1^iz_2^j \in H^2_{\mathcal{E}}(\D^2) : f_{ij}=0 \ \ \text{for} \ \ 0 \leq i \leq k-1 \right\}.
\]
Therefore, we have that 
\[
\bigcap_{k=1}^\infty T_1^k(H^2_{\mathcal{E}}(\D^2))=\{0\} \quad \text{and} \quad \bigcap_{k=1}^\infty T_2^k(H^2_{\mathcal{E}}(\D^2))=\{0\}.
\]
It follows from Theorem \ref{thm_Wold} that both $T_1$ and $T_2$ are pure isometries. Therefore, $(T_1, T_2)$ is a doubly $q$-commuting pair of c.n.u. contractions. \qed
\end{eg}

\begin{eg}
Let $\mathcal{E}$ be a Hilbert space and let $\alpha_1, \alpha_2, q$ be scalars satisfying $0<|\alpha_1|, |\alpha_2|<1$ and $|q|=1$. For the doubly $q$-commuting pair of pure isometries $(T_1, T_2)$ on $H^2_{\mathcal{E}}(\D^2)$ as in Example \ref{eg404}, we define $(A_1, A_2)=(\alpha_1T_1, \alpha_2T_2)$ on $H^2_{\mathcal{E}}(\D^2)$. It is easy to see that $(A_1, A_2)$ is a doubly $q$-commuting pair of contractions, and that 
\[
\|A_jf\|=|\alpha_j| \ \|T_jf\|=|\alpha_j|\|f\|<\|f\| \quad (j=1,2)
\]
for every $f \in H^2_{\mathcal{E}}(\D^2) \setminus \{0\}$. Thus, $(A_1, A_2)$ is a doubly $q$-commuting pair of c.n.i. contractions. \qed
\end{eg}	

Using the same algorithm as described in Section \ref{sec03}, we can easily achieve a Levan-type decomposition for any (infinite) family of doubly $Q$-commuting c.n.u. contractions so that each operator in the decomposition is a fundamental c.n.u. contraction. To avoid triviality, we assume that the family consists of infinitely many non-fundamental c.n.u. contractions.
Here we denote the limiting set as $\widetilde{\Delta_{\infty}}$, previously denoted as $\Delta_\infty$ for doubly $Q$-commuting contractions.
 
	\begin{thm}\label{Levan general}
		Let $\mathcal{C}=\{T_r: r \in R \}$ be a doubly $Q$-commuting family of c.n.u. contractions acting on a Hilbert space $\mathcal{H}$ which consists of infinitely many non-fundamental c.n.u. contractions. If $C(R)$ is the set of all functions $\varphi: R \rightarrow\{A_{shift}, A_{cni}\}$ and 
		$
		\widetilde{\Delta_\infty}=\{(\mathcal{C}_\varphi, \mathcal{H}_\varphi): \varphi \in C(R)\}$,	where $\mathcal{C}_\varphi=\{T_r|_{\mathcal{H}_\varphi}: r \in R\}$, then we have the following assertions.
		\begin{enumerate}
			\item  $\mathcal{H}$ admits an orthogonal decomposition $\mathcal{H}=\bigoplus_{\varphi \in C(R)} \mathcal{H}_\varphi$ such that each $\mathcal{H}_\varphi$ is a joint reducing subspace for $\mathcal{T}$ and $Q$.
			
\item  For each $\varphi \in C(R)$, the components of the operator tuple $\{T_r |_{\mathcal{H}_\varphi}: r \in R \}$ are all fundamental c.n.u. contractions.

\item  There is exactly one element, say $\left(\mathcal{C}_1, \mathcal{H}_1\right)$ in $\widetilde{\Delta_\infty}$ such that the components of $\mathcal{C}_1$ are all pure isometries, and this is determined by the constant function $\varphi_1: R \rightarrow\{A_{shift}, A_{cni}\}$ defined by $\varphi_1(r )=A_{shift}$ for all $r \in R$.

\item  There is exactly one element, say $\left(\mathcal{C}_{2^{|R|}}, \mathcal{H}_{2^{|R|}}\right)$ in $\widetilde{\Delta_\infty}$ such that the components of $\mathcal{C}_{2^{|R|}}$ are all c.n.i. contractions, and this is determined by the constant function $\varphi_{2^{|R|}}: R \rightarrow\{A_{shift}, A_{cni}\}$ defined by $\varphi_{2^{|R|}}(r )=A_{cni}$ for all $r \in R$.
	\smallskip 
	
\item  The cardinality of the set $\widetilde{\Delta_\infty}$ is $2^{|R|}$, where $|R|$ is the cardinality of the set $R$.
			
		\end{enumerate}
		One or more members of $\left\{\mathcal{H}_\varphi: \varphi \in C(R)\right\}$ may coincide with the trivial subspace $\{0\}$.
	\end{thm}
	
	\vspace{0.05cm}

\section{Decomposition of $Q$-commuting tuple of contractions}\label{sec05}
	
	\vspace{0.3cm}
	
	\noindent The decomposition results in Section \ref{sec03} are subjected to the condition of double $Q$-commutativity. In this Section, our aim is to find an analogue of an interesting result due to Burdak \cite{BurdakIII} in the $Q$-commuting setting. Before that we need the following terminology.
	
\begin{defn}
A commuting pair of contractions $(T_1, T_2)$ acting on a Hilbert space
$\HS$ is called a \textit{strongly completely non-unitary pair} if there
is no proper subspace $\mathcal{L}$ of $\HS$ that reduces both $T_1, T_2$ and atleast one of $T_1|_{\mathcal{L}}$, $T_2|_{\mathcal{L}}$ is a
unitary.
\end{defn}

\begin{thm}[\cite{BurdakIII}, Theorem 2.1]
         For a commuting pair of contractions $(T_1, T_2)$ acting on a Hilbert
space $\HS$, there is a unique decomposition
         \[
         \HS=\HS^{(uu)} \oplus \HS^{(uc)}\oplus \HS^{(cu)}\oplus \HS(s),
         \]
        where $\HS^{(uu)} , \HS^{(uc)}, \HS^{(cu)}, \HS(s)$ are maximal joint
reducing subspaces for $(T_1, T_2)$ such that the following hold:
        \begin{enumerate}
                        \item $T_{1}|_{\HS^{(uu)} }$ and $T_{2}|_{\HS^{(uu)} }$ are unitaries;
                        \item $T_{1}|_{\HS^{(uc)} }$ is a unitary and $T_{2}|_{\HS^{(uc)} }$ is
a c.n.u. contraction;
                        \item $T_{1}|_{\HS^{(cu)} }$ is a c.n.u. contraction and
$T_{2}|_{\HS^{(cu)} }$ is a unitary;
                        \item $(T_1|_{\HS(s)},T_2|_{\HS(s)})$ is a strongly completely
non-unitary pair of contractions.
                        \end{enumerate}

\end{thm}	

Note that the space $\HS(s)=\HS \ominus (\HS^{(uu)}\oplus\HS^{(uc)}\oplus \HS^{(cu)})$ need not reduce both $T_1, T_2$ to c.n.u. contractions. Also, $\HS(s)$ has no proper nonzero subspace $\mathcal{L}$ reducing $T_1, T_2$ such that atleast one of $T_1|_{\mathcal{L}}, T_2|_{\mathcal{L}}$ is a unitary. Taking cue from these two facts, we define the following classes of $Q$-commuting contractions. 

\begin{defn}
		A $Q$-commuting tuple $\underline{T}=(T_1, \dotsc, T_n)$ of contractions acting on a Hilbert space $\HS$ is said to be
		\begin{enumerate}
			\item \textit{strongly c.n.u.} if the existence of a joint reducing subspace $\mathcal{L}$ of $\underline{T}, Q$ such that atleast $(n-1)$ of $T_1|_\mathcal{L}, \dotsc, T_n|_\mathcal{L}$ are unitaries implies that $\mathcal{L}=\{0\}$;
			\item \textit{strongly pure isometry} if each $T_j$ is an isometry and the existence of a joint reducing subspace $\mathcal{L}$ of $\underline{T}, Q$ such that atleast $(n-1)$ of $T_1|_\mathcal{L}, \dotsc, T_n|_\mathcal{L}$ are unitaries implies that $\mathcal{L}=\{0\}$.
		\end{enumerate}
	\end{defn}

Let us give an example of a $Q$-commuting tuple of strongly pure isometry. 

\begin{eg}\label{eg502}
Let $U$ be a unitary on a Hilbert space $\mathcal E$. Define operators $T_1, T_2$ on $\ell^2(\mathcal E)$ as 
\[
T_1(a_0, a_1, a_2, \dotsc) = (0, a_0,Ua_1,U^2a_2, \dotsc), \quad  T_2(a_0, a_1, a_2, \dotsc) = (0, a_0, a_1, a_2, \dotsc).
\]
Evidently, $T_1$ and $T_2$ are isometries. For $Q : \ell^2(\mathcal E) \to \ell^2(\mathcal E)$ given
by $Q(a_0, a_1, \dotsc) = (Ua_0,Ua_1, \dotsc)$, we have that 
\[
Q^*Q=QQ^*=I, \quad T_1T_2 = QT_2T_1, \quad T_1Q = QT_1 \quad \text{and} \quad T_2Q = QT_2.
\]
A few steps of routine calculations lead to
\[
T_1^*(a_0, a_1, a_2, \dotsc) = (a_1, U^*a_2, {U^*}^2a_3, \dotsc) \quad \text{where} \quad (a_0, a_1, a_2, \dotsc) \in \ell^2(\mathcal E),
\]
and to the fact that ${T_j^*}^n \to 0$ as $n \to \infty$ strongly for $j= 1, 2$. Thus, $(T_1, T_2)$ is a $Q$-commuting tuple of pure isometries and thus, none of them is unitary on any reducing subspace of $\ell^2(\mathcal E)$. Hence, $(T_1,T_2)$ is a strongly pure isometry $Q$-commuting pair. In addition, $T_1T_2^*\ne Q^*T_2^*T_1$, because
\[
Q^*T_2^*T_1(a_0, 0, 0, \dotsc)-T_1T_2^*(a_0, 0, 0, \dotsc)=(U^*a_0, 0, 0, \dotsc) \quad (a_0 \in \mathcal{E}).
\]
Thus, $(T_1, T_2)$ also turns out to be a $Q$-commuting pair that is not doubly $Q$-commuting. \qed
\end{eg}

The following lemma is a natural analogue of Proposition \ref{normal} in the setting of $Q$-commuting contractions. However, in Proposition \ref{normal} we did not have any bound on the modulus of $q$. We skip the proof here as it can be easily imitated from the proof of Proposition \ref{normal}. 

\begin{lem}\label{lem612}
		Let $(T_1, T_2)$ be a $Q$-commuting pair of contractions acting on a Hilbert space $\HS$. If $T_1$ or $T_2$ is normal, then $T_1T_2^*=Q^*T_2^*T_1$ holds.
	\end{lem}	

Now we arrive at the main result of this Section.
	
\begin{thm}\label{thm501}
	Let $(T_1, \dotsc, T_n)$ be a $Q$-commuting tuple of contractions acting on a Hilbert space $\mathcal{H}$. Then $\HS$ admits an orthogonal decomposition $\mathcal{H}=\HS(s) \oplus \left( \underset{1 \leq j \leq n+1}{\oplus}\HS_j\right)$ such that the following hold.
	\begin{enumerate}
		\item $\HS_1, \dotsc, \HS_{n+1}, \HS(s)$ are joint reducing subspaces for $T_1, \dotsc, T_n$ and $Q$.
		
		\item $T_i|_{\HS_j}$ is either a unitary or a c.n.u. contraction for $ 1 \leq i \leq n$ and $1 \leq j \leq n+1$.
			
			\item If $\widetilde \Theta_n$ is the set of all functions $\widetilde \theta: \{1, \dotsc, n\} \to \{A_u, A_{cnu}\}$ such that there are atleast $(n-1)$ values as $A_u$ among $\widetilde \theta(1), \dotsc, \widetilde \theta(n)$, then for every $\widetilde \theta \in \widetilde \Theta_n$, there is a unique subspace $\HS_{t} \ (1 \leq t \leq n+1)$ such that the $k$-th component of $(T_1|_{\HS_{t}}, \dotsc, T_n|_{\HS_{t}})$ is of the type $\widetilde \theta(k)$ for $1 \leq k \leq n$ and $\HS_t$ is the largest such joint reducing subspace for $T_1, \dotsc, T_n$ and $Q$.
				
	\item $(T_1|_{\mathcal{H}(s)}, \dotsc, T_n|_{\mathcal{H}(s)})$ is a strongly c.n.u. $Q$-commuting tuple of contractions.
	\end{enumerate}
	One or more subspaces in the decomposition  may coincide with the subspace $\{0\}$.
\end{thm}

\begin{proof}
	 		Let $Q=\{Q_{ij} : 1 \leq i < j \leq n\}$ and let $\HS^{(dc)}$ denote the following subspace:
			\[
		\overline{\mbox{span}}\bigg\{\mathcal{H}_0 \subseteq \underset{1 \leq i < j \leq n}{\bigcap} \ \mbox{Ker}(T_iT_j^*-Q_{ij}^*T_j^*T_i):  P_{\mathcal{H}_0}T_i=T_iP_{\mathcal{H}_0},  P_{\mathcal{H}_0}Q_{ij}=Q_{ij}P_{\mathcal{H}_0} \bigg\}.
		\]
		Evidently, $\HS^{(dc)}$ is the largest joint reducing subspace for $T_1, \dotsc, T_n$ and $Q$ such that the tuple
		$
		(T_1|_{\HS^{(dc)}}, \dotsc, T_n|_{\HS^{(dc)}})
		$
		consists of doubly $Q$-commuting contractions. We have by Theorem \ref{thm303F} that $\mathcal{H}^{(dc)}$ admits a unique decomposition into an orthogonal sum of $2^n$ joint reducing subspaces, say $\HS^{(dc)}_1, \dotsc, \HS^{(dc)}_{2^n}$ of $T_1, \dotsc, T_n$ and $Q$ such that the following hold.
	
		\begin{enumerate}
			\item[(i)] $T_i|_{\HS^{(dc)}_j}$ is either a unitary or a c.n.u. contraction for $ 1 \leq i \leq n$ and $1 \leq j \leq 2^n$.
			
			\item[(ii)] If $\Theta_n$ is the set of all functions $\theta: \{1, \dotsc, n\} \to \{A_u, A_{cnu}\}$, then for every $\theta \in \Theta_n$, there corresponds a unique subspace $\HS_{t(\theta)}$ with $1 \leq t(\theta) \leq 2^n$, such that the $k$-th component of $(T_1|_{\HS^{(dc)}_{t(\theta)}}, \dotsc, T_n|_{\HS^{(dc)}_{t(\theta)}})$ is of the type $\theta(k)$ for $1 \leq k \leq n$ and $\HS^{(dc)}_{t(\theta)}$ is the largest such joint reducing subspace for $T_1, \dotsc, T_n, Q$.
			
			\item[(iii)] $T_1|_{\HS^{(dc)}_1}, \dotsc, T_n|_{\HS^{(dc)}_1}$ are unitaries and $T_1|_{\HS^{(dc)}_{2^n}}, \dotsc, T_n|_{\HS^{(dc)}_{2^n}}$ are c.n.u. contractions.
\end{enumerate}

Let $\theta_1(k)=A_u$ for $1 \leq k \leq n$. There are only $n$ number of maps $\theta_2, \dotsc, \theta_{n+1}$ in $\Theta_n$ which take the value $A_{cnu}$ exactly at one point. For every $t \in \{1, \dots , n+1\}$, there is a unique maximal joint reducing subspace $\HS_t^{(dc)} \subset \HS^{(dc)}$ of $T_1, \dotsc, T_n, Q$ such that the $k$-th component of $(T_1|_{\HS_{t}}, \dotsc, T_{n}|_{\HS_{t}})$ is of the type $\theta_t(k)$ for $1 \leq k \leq n$. We now show that the subspace $\HS_t^{(dc)}$ is also maximal in $\HS$ in the above sense. Let $\HS_t \subseteq \HS$ be maximal in the above sense. It follows from the definition of $\theta_t$ that atleast $(n-1)$ among $T_1|_{\HS_t}, \dotsc, T_n|_{\HS_t}$ are unitaries. We have by Lemma \ref{lem612} that $(T_1|_{\HS_{t}}, \dotsc, T_{n}|_{\HS_{t}})$ is doubly $Q$-commuting and so, $\HS_t \subset \HS^{(dc)}$. By maximality, it follows that  $\HS_t = \HS_t^{(dc)}$. Consider the following closed joint reducing subspace of $T_1, \dotsc, T_n, Q$ : 
\[
\HS(s)=\HS \ominus (\HS_1 \oplus \dotsc \oplus \HS_{n+1}).
\]
We show that $(T_1|_{\mathcal{H}(s)}, \dotsc, T_n|_{\mathcal{H}(s)})$ is a strongly c.n.u. $Q$-commuting tuple of contractions. Let if possible, there be a closed subspace $\mathcal{L} \subseteq \mathcal{H}(s)$ reducing each $T_j$ and $Q$ such that atleast $(n-1)$ contractions among $T_1|_\mathcal{L}, \dotsc, T_n|_\mathcal{L}$ are unitaries. Without loss of generality, we assume that $T_1|_\mathcal{L}, \dotsc, T_{n-1}|_\mathcal{L}$ are unitaries. We show that 
		$(T_1|_\mathcal{L}, \dotsc,  T_n|_\mathcal{L})
		$
		is a doubly $Q$-commuting tuple. Let $(T_1', \dotsc, T_n')=(T_1|_\mathcal{L}, \dotsc, T_n|_\mathcal{L})$. Since $\mathcal{L}$ reduces $Q$, we have that $T_1', \dotsc, T_{n-1}'$ are normal operators acting on $\mathcal{L}$ such that $T_i'T_j'=Q_{ij}T_j'T_i'$ for $1 \leq i < j \leq n-1$. It follows from Lemma \ref{lem612} that $T_i'T_j'^*=Q_{ij}^*T_j'^*T_i'$. For $1 \leq j \leq n-1$, we have that
		$
		T_jT_n=Q_{jn}T_nT_j$ and thus $ T_j'T_n'=Q_{jn}T_n'T_j'$. Since each $T_j'$ is normal, we have by Lemma \ref{lem612} that $T_j'T_n'^*=Q_{jn}^*T_n'^*T_j'$ for $1 \leq j \leq n-1$. Consequently, we get that $(T_1', \dotsc, T_n')$ is a doubly $Q$-commuting tuple of contractions on $\mathcal{L}$ and hence, $\mathcal{L} \subseteq \mathcal{H}^{(dc)}$. We have by Theorem \ref{thm104} that $\mathcal{L}$ admits an orthogonal decomposition into two closed subspaces reducing $T_n'$, say $\mathcal{L}= \mathcal{L}_{u} \oplus \mathcal{L}_{c}$, such that $T_n'|_{\mathcal{L}_{u}}$ is a unitary and $T_n'|_{\mathcal{L}_{c}}$ is a c.n.u. contraction. It follows from Theorem \ref{Qreducing} that $\mathcal{L}_u, \mathcal{L}_c$ are joint reducing subspaces for $Q$ and  $T_1', \dotsc, T_{n-1}'$. Since $\HS^{(dc)}_{1}$ is the largest closed subspace of $\HS^{(dc)}$ reducing each $T_j$ to a unitary, we have $\mathcal{L}_u \subseteq \HS^{(dc)}_{1}$. Similarly, using the maximal property that $\mathcal{H}^{(dc)}_{\ell}$ (for $ 1 \leq \ell \leq n+1$) is the largest closed subspace of $\mathcal{H}^{(dc)}$ reducing each $T_1, \dotsc, T_{n-1}$ to a unitary and $T_n$ to a c.n.u. contraction, we have that $\mathcal{L}_c \subseteq \mathcal{H}^{(dc)}_{\ell}$. Consequently, we have 
		\[
		\mathcal{L}=\mathcal{L}_u \oplus \mathcal{L}_c \subseteq \mathcal{H}^{(dc)}_{1} \oplus \mathcal{H}^{(dc)}_{\ell} =  \mathcal{H}_{1} \oplus \mathcal{H}_{\ell} \subseteq \HS \ominus \HS(s) \quad \text{and so,} \quad 		\mathcal{L}=\{0\}.
		\] 
Thus, $(T_1|_{\mathcal{H}_{2^n}}, \dotsc, T_n|_{\mathcal{H}_{2^n}})$ is a strongly c.n.u. $Q$-commuting tuple of contractions. The uniqueness follows from the maximality of each subspace $\HS_{t}$ for $1 \leq t \leq n+1$ and the proof is complete.  
\end{proof}

An immediate consequence of the above result is the following decomposition theorem for a $Q$-commuting tuple of isometries.
	
\begin{thm}\label{prop503}
		Let $(V_1, \dotsc, V_n)$ be a $Q$-commuting tuple of isometries acting on a Hilbert space $\mathcal{H}$. Then $\HS$ admits an orthogonal decomposition $\mathcal{H}=\HS(s) \oplus \left( \underset{1 \leq j \leq n+1}{\oplus}\HS_j\right)$ such that the following statements hold.
	\begin{enumerate}
		\item $\HS_1, \dotsc, \HS_{n+1}, \HS(s)$ are joint reducing subspaces for $V_1, \dotsc, V_n$ and $Q$.
		
		\item $V_i|_{\HS_j}$ is either a unitary or a pure isometry for $ 1 \leq i \leq n$ and $1 \leq j \leq n+1$.
		
			\item If $\widetilde \Theta_n$ is the set of all functions $\widetilde \theta: \{1, \dotsc, n\} \to \{A_u, A_{shift}\}$ such that there are atleast $(n-1)$ values as $A_{u}$ among $\widetilde \theta(1), \dotsc, \widetilde \theta(n)$, then for every $\widetilde \theta \in \widetilde \Theta_n$, there is a unique maximal joint reducing subspace $\HS_{t} \ (1 \leq t \leq n+1)$ for $V_1, \dotsc, V_n, Q$ such that the $k$-th component of $(V_1|_{\HS_{t}}, \dotsc, V_n|_{\HS_{t}})$ is of the type $\widetilde \theta(k)$ for $1 \leq k \leq n$.
			
	\item $(V_1|_{\mathcal{H}(s)}, \dotsc, V_n|_{\mathcal{H}(s)})$ is a strongly pure isometry $Q$-commuting tuple.
	\end{enumerate}
One or more subspaces in the decomposition  may coincide with the subspace $\{0\}$.
		
	\end{thm}	

\vspace{0.05cm}

\section{Decomposition of a pair of non-commuting contractions}\label{sec07}
	
		\vspace{0.3cm}
	
	\noindent In this Section, we study a more general class of a pair of contractions (not necessarily commuting). For a $Q$-commuting pair of contractions $(T_1, T_2)$ acting on a Hilbert space $\HS$, there is a unitary $Q$ on $\HS$ such that $T_1T_2=QT_2T_1$. In this Section, we study a pair of contractions with different operator equations involving $Q$, e.g. $T_1T_2=T_2QT_1$ or $T_1T_2=T_2T_1Q$. We prove the decomposition results for such classes of non-commuting pair of contractions $(T_1, T_2)$ under the weaker hypothesis that $Q$ commutes with $T_1T_2$ (not necessarily with $T_1$ and $T_2$).

	\begin{lem}\label{P_unitary}
		Let $(A,B)$ be a pair of contractions on a Hilbert space $\mathcal{H}$ and let $Q$ be a unitary on $\HS$ such that any of the following relations holds:
		\[
		AB=QBA, \quad AB=BQA \quad \mbox{or} \quad AB=BAQ.
		\]	
		Then $AB$ is unitary if and only if both $A$ and $B$ are unitaries.
	\end{lem} 
	
\begin{proof}
		If $A$ and $B$ both are unitaries, then it is clear that $AB$ is a unitary. Conversely, let $AB$ be a unitary and let us denote $P=AB$. Then $A$ is surjective and $B$ is injective. We claim that $A$ and $B$ have bounded inverses. To prove this, we proceed case by case depending on the positions of $Q$. 
		
		\vspace{0.2cm}
		
		\noindent \textit{Case 1:} Let $AB=QBA$. Then $BA=Q^*AB=Q^*P$  and so, $BA$ is a unitary as well. Thus $A$ is injective and hence, bijective. It follows from bounded inverse theorem that $A$ has a bounded inverse and so, $B=Q^*PA^{-1}$ is invertible too.
		
		\vspace{0.2cm}
		
		\noindent \textit{Case 2:} Let $AB=BQA$. Take any $x \in Ker(A)$. Then  
		\[
		Px=ABx=BQAx=0 \quad \text{and so,} \quad x=0
		\] 
		which shows that $A$ is injective. Consequently, $A$ is a bijective operator. By bounded inverse theorem, $A$ has a bounded inverse and hence, $B=PA^{-1}Q^*$ is invertible.
		
		\vspace{0.2cm}
		
		\noindent \textit{Case 3:} Let $AB=BAQ$. Then $BA=ABQ^*=PQ^*$ is a unitary implying that $A$ is injective. Consequently, $A$ is bijective and by bounded inverse theorem, $A$ has a bounded inverse. Thus, $B=PQ^*A^{-1}$ is invertible.
		
		\vspace{0.2cm}
		
		\noindent In either case, the inverses of $A$ and $B$ exist and are bounded. Take any $x \in \mathcal{H}$. Then
		\[
		\|x\|=\|BB^{-1}x\|\leq \|B^{-1}x\| \quad \text{and so,} 
		\quad 
		\|Ax\|=\|ABB^{-1}x\|=\|PB^{-1}x\|=\|B^{-1}x\| \geq \|x\|.
		\]
		Since $A$ is a contraction, we get that $\|Ax\|=\|x\|$ for every $x\in \mathcal{H}$. Thus $A$ is a surjective isometry and hence, a unitary. Similarly, we can prove that $B$ is a unitary. The proof is complete.
	\end{proof}

The following result is crucial and will be used in sequel.

\begin{prop}[\cite{Dou:Muh:Pea}, Proposition 2.2]\label{Douglas}
		Let $A, B$ be contractions acting on Hilbert spaces $\mathcal{H}_1$ and $\mathcal{H}_2$ respectively and let $X \in \mathcal{B}(\mathcal{H}_2, \mathcal{H}_1).$ The operator on $\mathcal{H}_1\oplus \mathcal{H}_2$ defined by 
		\[
		\begin{bmatrix}
			A & X\\
			0 & B\\
		\end{bmatrix}
		\]
		is a contraction if and only if there is a contraction $C \in \mathcal{B}(\mathcal{H}_2, \mathcal{H}_1)$ such that $X=D_{A^*}CD_{B}.$	
	\end{prop}

\begin{lem}\label{FoAB=QBA}
Let $(A,B)$ be a pair of contractions acting on a Hilbert space $\mathcal{H}$ with $P=AB$ and let $Q$ be any unitary on $\HS$. Then the following hold. 
\begin{enumerate} 
\item If $AB=QBA$, then there exist operators $F_1, F_2$ in $\mathcal{B}(\mathcal{D}_P)$ such that
\[
 D_B^2A=D_PF_1D_P \quad \text{and} \quad D_A^2B=D_PF_2D_P.
 \]

\item If $AB=BAQ$, then there exist operators $F_1, F_2$ in $\mathcal{B}(\mathcal{D}_P)$ such that
\[
D_B^2AQ=D_PF_1D_P \quad \text{and} \quad Q^*D_A^2B=D_PF_2D_P.
\] 

\item If $AB=BQA$, then there exist operators $F_1, F_2$ in $\mathcal{B}(\mathcal{D}_P)$ such that 
\[
 D_B^2QA=D_PF_1D_P \quad \text{and} \quad D_A^2B=D_PF_2D_P.
 \]
 \end{enumerate}
\end{lem}

 \begin{proof} 
 We divide the proof into three cases depending on the positions of $Q$ in the $Q$-commuting relations. 
 
 \smallskip
 
 \noindent \textbf{Case 1.} Assume that $AB=QBA$. It follows from Proposition \ref{Douglas} that $ \begin{bmatrix} B^* & D_B^2\\ 0 & B\\ \end{bmatrix} $ acting on $\mathcal{H}\oplus \mathcal{H}$ is a contraction. Then the operator $X_1$, where
 \[ 
 X_1=\begin{bmatrix} QP^* & QD_B^2A\\ 0 & P\\ \end{bmatrix}= \begin{bmatrix} QB^*A^* & QD_{B}^2A\\ 0 & QBA\\ \end{bmatrix}= \begin{bmatrix} Q & 0\\ 0 & Q\\ \end{bmatrix}\begin{bmatrix} B^* & D_{B}^2\\ 0 & B\\ \end{bmatrix}\begin{bmatrix} A^* & 0\\ 0 & A\\ \end{bmatrix}, 
 \]
 is also a contraction on $\mathcal{H}\oplus \mathcal{H}$. Again by Proposition \ref{Douglas}, there exists $G_1 \in \mathcal{B}(\mathcal{H})$ such that $QD_B^2A=D_{PQ^*}G_1D_P$. Since $Q$ is a unitary, we have that $D^2_{PQ^*}=I-QP^*PQ^*=QD_P^2Q^*$. By spectral theorem, $D_{PQ^*}=QD_PQ^*$. Thus \begin{equation}\label{D_PQ*} D_B^2A=Q^*D_{PQ^*}G_1D_P=D_PQ^*G_1D_P=D_{AB}Q^*G_1D_{AB}. \end{equation} Since $BA=Q^*AB$, we interchange $A, B$ and replace $Q$ by $Q^*$ in (\ref{D_PQ*}) and obtain $G_2 \in \mathcal{B}(\mathcal{H})$ such that $D_A^2B=D_{BA}QG_2D_{BA}$. Since $Q$ is a unitary, we have $D_{BA}=D_{Q^*P}=D_P$ and hence, \begin{equation}\label{D_PQ} 
 \quad D_A^2B=D_{BA}QG_2D_{BA}=D_{P}QG_2D_{P}. 
 \end{equation}
 Let us define
 $ F_1:=P_{\mathcal{D}_P}Q^*G_1|_{\mathcal{D}_P}$ and $F_2:=P_{\mathcal{D}_P}QG_2|_{\mathcal{D}_P}$,
 where $P_{\mathcal{D}_P}$ is the projection of $\HS$ onto $\mathcal{D}_P$. Since $D_PP_{\mathcal{D}_P}=P_{\mathcal{D}_P}D_P=D_P$, we have that \[ D_PF_1D_P=D_PP_{\mathcal{D}_P}Q^*G_1D_P=D_PQ^*G_1D_P=D_B^2A \quad [\text{by (\ref{D_PQ*})}] \] and \[ D_PF_2D_P=D_PP_{\mathcal{D}_P}QG_2D_P=D_PQG_2D_P=D_A^2B \qquad [\text{by (\ref{D_PQ})}]. \] 
 
 \medskip
 
  \noindent \textbf{Case 2.} We now assume that $AB=BAQ$. It follows from Proposition \ref{Douglas} that $ \begin{bmatrix} B^* & D_B^2\\ 0 & B\\ \end{bmatrix} $ is a contraction on $\mathcal{H}\oplus \mathcal{H}$. Then the operator $X_2$, where
   \[ 
  X_2= \begin{bmatrix} P^* & D_B^2AQ\\ 0 & P\\ \end{bmatrix}= \begin{bmatrix} B^*A^* & D_B^2AQ\\ 0 & BAQ\\ \end{bmatrix}= \begin{bmatrix} B^* & D_B^2\\ 0 & B\\ \end{bmatrix}\begin{bmatrix} A^* & 0\\ 0 & AQ\\ \end{bmatrix}, 
   \] 
   is also a contraction on $\mathcal{H}\oplus \mathcal{H}$. Again by Proposition \ref{Douglas}, there exists $G_1 \in \mathcal{B}(\mathcal{H})$ such that \begin{equation}\label{eqn703} D_B^2AQ=D_{AB}G_1D_{AB}=D_{P}G_1D_P. \end{equation} Since $BA=ABQ^*$, one can interchange $A, B$ and replace $Q$ by $Q^*$ in (\ref{eqn703}) and obtain $G_2 \in \mathcal{B}(\mathcal{H})$ such that $D_A^2BQ^*=D_{BA}G_2D_{BA}$. Since $Q$ is a unitary, we have that $D_{BA}^2=D_{PQ^*}^2=QD_P^2Q^*$. By spectral theorem, $D_{BA}=QD_PQ^*$. Thus $D_{A}^2BQ^*=QD_{P}Q^*G_2QD_{P}Q^*$ and so, \begin{equation}\label{eqn704} Q^*D_{A}^2B=D_{P}Q^*G_2QD_{P}. \end{equation}
   Let us define $ F_1:=P_{\mathcal{D}_P}G_1|_{\mathcal{D}_P}$ and $F_2:=P_{\mathcal{D}_P}Q^*G_2Q|_{\mathcal{D}_P}$.
  Since $D_PP_{\mathcal{D}_P}=P_{\mathcal{D}_P}D_P=D_P$, we have by (\ref{eqn703}) and (\ref{eqn704}) that $ D_PF_1D_P=D_PG_1D_P=D_B^2AQ$ and $D_PF_2D_P=D_PQ^*G_2QD_P=Q^*D_A^2B$ respectively. 
  
  \medskip 
  
  \noindent \textbf{Case 3.} Let us assume that $AB=BQA$. It follows from Proposition \ref{Douglas} that $ \begin{bmatrix} B^* & D_B^2\\ 0 & B\\ \end{bmatrix} $ is a contraction on $\mathcal{H}\oplus \mathcal{H}$. Then the operator $X_3$, where
   \[
   X_3=\begin{bmatrix} P^* & D_B^2QA\\ 0 & P\\ \end{bmatrix}= \begin{bmatrix} B^*A^* & D_B^2QA\\ 0 & BQA\\ \end{bmatrix}= \begin{bmatrix} B^* & D_B^2\\ 0 & B\\ \end{bmatrix}\begin{bmatrix} A^* & 0\\ 0 & QA\\ \end{bmatrix},
    \]
    is also a contraction on $\mathcal{H}\oplus \mathcal{H}$. By Proposition \ref{Douglas}, there exists $G_1 \in \mathcal{B}(\mathcal{H})$ such that \begin{equation}\label{eqn705} D_B^2QA=D_{AB}G_1D_{AB}=D_{P}G_1D_P. \end{equation} Let $S=QB$ and $R=QA$. Then $ RQ^*S=QAB=QBQA=SR. $ One can replace $A, B$ and $Q$ with $S, R$ and $Q^*$ respectively in (\ref{eqn705}) and obtain $G_2 \in \mathcal{B}(\mathcal{H})$ such that $ D_R^2Q^*S=D_{SR}G_2D_{SR} . $ Note that \[ D_{R}^2Q^*S=(I-R^*R)Q^*S=(I-A^*Q^*QA)Q^*QB=D_A^2B \] and \[ D_{SR}^2=I-(SR)^*SR=I-(QBQA)^*QBQA=I-(AB)^*AB=D_P^2. \] By spectral theorem, $D_{SR}=D_P$ and consequently, we have that \begin{equation}\label{eqn706}
     \begin{split} D_A^2B=D_{R}^2Q^*S=D_{SR}G_2D_{SR}=D_PG_2D_P. \end{split} 
    \end{equation} 
    Let us define $F_1:=P_{\mathcal{D}_P}G_1|_{\mathcal{D}_P}$ and $ F_2:=P_{\mathcal{D}_P}G_2|_{\mathcal{D}_P}$. It follows from (\ref{eqn705}) \& (\ref{eqn706}) that $D_B^2QA=D_PF_1D_P$ and $D_A^2B=D_PF_2D_P$. The proof is complete. \end{proof}

The next lemma is an important result which we shall use later.

	\begin{lem}\label{TUcnu}
		Let $T$ be a contraction acting on a Hilbert space $\mathcal{H}$ and let $U$ be a unitary on $\mathcal{H}$ such that $TU=UT$. Then $T$ is a c.n.u. contraction if and only if $TU$ is a c.n.u. contraction.
	\end{lem}
	
	\begin{proof}
		Let $T$ be a c.n.u. contraction. Since $TU$ is a contraction on $\mathcal{H},$ it follows from Theorem \ref{thm104} that $\mathcal{H}$ admits an orthogonal decomposition into two closed reducing subspaces $\mathcal{H}_1$ and $\mathcal{H}_2$ for $TU$ such that $TU|_{\mathcal{H}_1}$ is a unitary and $TU|_{\mathcal{H}_2}$ is a c.n.u. contraction. Moreover, the space $\mathcal{H}_1$ is given by 
		\[
		\mathcal{H}_1=\{x \in \mathcal{H}: \|(TU)^nx\|=\|x\|=\|(TU)^{*n}x\|, \ n \in \mathbb{N}\}.
		\]
		Since $U$ commutes with $T$, we have that	
		\[
		\HS_1=\{x \in \mathcal{H}: \|T^nx\|=\|x\|=\|T^{*n}x\|, \ n \in \mathbb{N}\}.
		\]
Again, by Theorem \ref{thm104}, $\HS_1$ is the maximal closed subspace of $\mathcal{H}$ reducing $T$ and on which $T$ acts as a unitary. Since $T$ is a c.n.u. contraction , we have that $\mathcal{H}_1=\{0\}$. Consequently, $\mathcal{H}=\mathcal{H}_2$ and the desired conclusion follows. To see the converse, let $T'=TU$ and let $T'$ be a c.n.u. contraction. Then $T'U^*=U^*T'=T$ and by the previous case, $T\; (=T'U^*)$ is a c.n.u. contraction. 
	\end{proof}
	
\begin{rem}\label{rem707}
		Let $A, B$ be contractions (not necessarily commuting) acting on a Hilbert space $\HS$ such that any of the following relations holds: 
		\[
		AB=QBA, \quad AB=BQA, \quad  AB=BAQ,
		\]
		for some unitary $Q \in \mathcal B(\HS)$. Let $P=AB$ and $PQ=QP$. Since $P$ is a contraction, $P$ admits a canonical decomposition say, $P_1 \oplus P_2$ with respect to the orthogonal decomposition $\HS=\HS_1 \oplus \HS_2$. Since $Q$ is a unitary, Fuglede's theorem \cite{Fuglede} implies that $Q$ doubly commutes with $P$. It follows from Theorem \ref{Qreducing} that $\HS_1$ and $\HS_2$ reduce $Q$.
	\end{rem}

Actually, we want to show that the subspaces $\HS_1, \HS_2$ as in Remark \ref{rem707} reduce $A$ and $B$ as well.

\begin{thm}\label{reducingII} Let $(A,B)$ be a pair of contractions acting on a Hilbert space $\HS$ and let $Q$ be any unitary in $\mathcal B(\HS)$ such that any of the following relations holds: 
\[ AB=QBA, \quad AB=BQA \quad \text{or} \quad AB=BAQ. 
\]
Let $P=AB$ and let $Q$ commute with $P$. If $P=P_1 \oplus P_2$ is the canonical decomposition of $P$ with respect to the orthogonal decomposition $\HS=\HS_1 \oplus \HS_2$, then $\mathcal{H}_1, \mathcal{H}_2$ reduce $A, B$ and $Q$. 
\end{thm}

\begin{proof} It follows from Remark \ref{rem707} that $\HS_1, \HS_2$ reduce $Q$. To avoid triviality, we assume that $P$ is neither a unitary nor a c.n.u. contraction. With respect to $\mathcal{H}=\mathcal{H}_1\oplus \mathcal{H}_2,$ let \[ A=\begin{bmatrix} A_{11} & A_{12}\\ A_{21} & A_{22}\\ \end{bmatrix}, \quad B=\begin{bmatrix} B_{11} & B_{12}\\ B_{21} & B_{22}\\ \end{bmatrix}, \quad Q=\begin{bmatrix} Q_1 & 0 \\ 0& Q_2\\ \end{bmatrix} \quad \text{and} \quad P=\begin{bmatrix} P_1 & 0\\ 0 & P_2\\ \end{bmatrix}. \] Note that $P_1$ is a unitary and $P_2$ is a c.n.u. contraction. We split the proof into three parts corresponding to three different positions of $Q$ in the $Q$-commuting relations. 

\medskip 

\noindent \textbf{Case 1.} Let $AB=QBA$. It follows from Part $(1)$ of Lemma \ref{FoAB=QBA} that there exist operators $F_1$ and $F_2$ acting on $\mathcal{D}_P$ such that $ D^2_BA=D_PF_1D_P$ and $D_A^2B=D_PF_2D_P$. It is easy to see that \[ D_B^2A=A-B^*BA=A-B^*Q^*AB=A-B^*Q^*P \quad \mbox{and} \quad D_A^2B=B-A^*AB=B-A^*P. \] Consequently, we have that \[ A-B^*Q^*P=D_PF_1D_P \quad \text{and} \quad B-A^*P=D_PF_2D_P. \] With respect to the decomposition $\mathcal{D}_P=\mathcal{D}_{P_1} \oplus \mathcal{D}_{P_2}=\{0\}\oplus \mathcal{D}_{P_2},$ let \[ F_1=\begin{bmatrix} 0 & 0\\ 0 & F_1' \end{bmatrix} \quad \text{and} \quad F_2=\begin{bmatrix} 0 & 0\\ 0 & F_2' \end{bmatrix}. \] Then, from $A-B^*Q^*P=D_PF_1D_P$, we have \begin{align} A_{11} &=B_{11}^*Q_1^*P_1, \quad A_{12} =B_{21}^*Q_2^*P_2, \label{eqn1.2} \\ A_{21} &=B_{12}^*Q_1^*P_1, \quad A_{22}=B_{22}^*Q_2^*P_2+D_{P_2} F_{1}' D_{P_2} \label{eqn2.2}. \end{align} Similarly, from $B-A^* P=D_P F_2 D_P$, we have \begin{align} B_{11} &=A_{11}^* P_1, \quad B_{12} =A_{21}^* P_2, \label{eqn3.2}\\ B_{21} &=A_{12}^* P_1, \quad B_{22} =A_{22}^* P_2 + D_{P_2} F_{2}' D_{P_2} \label{eqn4.2}. \end{align} Since $APQ^*=PA$, we have 
\begin{align} A_{11} P_1Q_1^* &=P_1 A_{11}, \quad A_{12}P_2Q_2^*=P_1 A_{12}, \label{eqn5.2} \\ A_{21} P_1 Q_1^*&=P_2 A_{21}, \quad A_{22} P_2Q_2^*=P_2 A_{22} \label{eqn6.2}. 
\end{align}
Also, from $QBP=PB,$ we have 
\begin{align} Q_1B_{11} P_1 &=P_1 B_{11}, \quad Q_1B_{12} P_2=P_1 B_{12}, \label{eqn7.2} \\ Q_2B_{21} P_1 &=P_2 B_{21}, \quad Q_2B_{22} P_2=P_2 B_{22}. \label{eqn8.2} 
\end{align}
Since $PQ=QP$, we have by (\ref{eqn3.2}), (\ref{eqn2.2}) and (\ref{eqn6.2}) that
\begin{equation}\label{eqn9.2} P_2 P_2^* A_{21}=P_2 B_{12}^*=P_2A_{21}P_1^*Q_1=A_{21}. \end{equation}
It follows from (\ref{eqn6.2}), (\ref{eqn3.2}) and (\ref{eqn2.2}) that 
\begin{equation}\label{eqn10.2} P_2^* P_2 A_{21}=P_2^*A_{21}P_1Q_1^*=B_{12}^*P_1Q_1^*=A_{21}.
\end{equation}
 By (\ref{eqn6.2}), the range of $A_{21}$ is invariant under $P_2$ and $ P_2^* A_{21}P_1=P_2^* P_2 A_{21}Q_1=A_{21}Q_1. $ Thus, the range of $A_{21}$ is invariant under $P_2^*$ since $P_1$ is a unitary acting on $\HS_1$. It follows from (\ref{eqn9.2}) and (\ref{eqn10.2}) that $P_2$ is a unitary on the range of $A_{21}$. Hence, $A_{21}=0$ as $P_2$ is a c.n.u. contraction.

 \medskip 
 
Next we show that $B_{21}=0$. Let $R=Q_2^*P_2$. Then, $R=P_2Q_2^*$ as $P_2$ and $Q_2$ doubly commute. It follows from Lemma \ref{TUcnu} that $R$ is a c.n.u. contraction. From (\ref{eqn8.2}), (\ref{eqn1.2}) and (\ref{eqn4.2}), we have
\begin{equation}\label{eqn11.2} R^*RB_{21}=P_2^*P_2B_{21}=P_2^*Q_2B_{21}P_1=A_{12}^*P_1=B_{21}. 
\end{equation}

 Using (\ref{eqn1.2}), (\ref{eqn4.2}) and (\ref{eqn8.2}), we have \begin{equation}\label{eqn12.2} RR^*B_{21}=P_2P_2^*B_{21}=P_2Q_2^*A_{12}^*=P_2Q_2^*B_{21}P_1^*=B_{21}. \end{equation} It follows from (\ref{eqn8.2}) that \begin{equation}\label{eqn13.2} RB_{21}=Q_2^*P_2B_{21}=B_{21}P_1 \end{equation} and so, \begin{equation}\label{eqn14.2} R^*B_{21}P_1=R^*RB_{21}=B_{21}, 
 \end{equation}
 where the last equality follows from (\ref{eqn11.2}). Since $P_1$ is a unitary, we have by (\ref{eqn13.2}) and (\ref{eqn14.2}) that the range of $B_{21}$ is invariant under $R$ and $R^*$ respectively. By (\ref{eqn11.2}) \& (\ref{eqn12.2}), we have that $R$ acts as a unitary of the range of $B_{21}.$ Since $R$ is a c.n.u. contraction, $B_{21}=0$. It follows from (\ref{eqn1.2}) and (\ref{eqn3.2}) that $A_{12}=0$ and $B_{12}=0$ respectively. Thus, $\mathcal{H}_1$ and $\mathcal{H}_2$ reduce $A, B$. 
 
 \medskip 
 
 \noindent \textbf{Case 2.} Let $AB=BAQ$. It follows from Fuglede's Theorem \cite{Fuglede} that $P$ and $Q$ doubly commute. Then 
 \[ Q^*AB=Q^*P=PQ^*=ABQ^*=BAQQ^*=BA \quad \text{and so,} \quad AB=QBA. 
 \]
 The desired conclusion follows from Case 1.
 
  \medskip 
  
  \noindent \textbf{Case 3.} Let $AB=BQA$. By Part $(3)$ of Lemma \ref{FoAB=QBA}, there exist operators $F_1$ and $F_2$ on $\mathcal{D}_P$ such that $ D^2_BQA=D_PF_1D_P$ and $D_A^2B=D_PF_2D_P$. Then 
  \[ 
  D_B^2QA=QA-B^*BQA=QA-B^*AB=QA-B^*P \quad \mbox{and} \quad D_A^2B=B-A^*AB=B-A^*P. 
  \] 
  Thus, we have that $ QA-B^*P=D_PF_1D_P$ and $ B-A^*P=D_PF_2D_P$. With respect to the decomposition $\mathcal{D}_P=\mathcal{D}_{P_1} \oplus \mathcal{D}_{P_2}=\{0\}\oplus \mathcal{D}_{P_2},$ let $ F_i=\begin{bmatrix} 0 & 0\\ 0 & F_i' \end{bmatrix}$ for $i=1,2$. From $QA-B^*P=D_PF_1D_P$, we have \begin{align} Q_1A_{11} &=B_{11}^*P_1, \quad Q_1A_{12} =B_{21}^*P_2, \label{eqn1.3} \\ Q_2A_{21} &=B_{12}^*P_1, \quad Q_2A_{22}=B_{22}^*P_2+D_{P_2} F_{1}' D_{P_2} \label{eqn2.3}. \end{align} Similarly, from $B-A^* P=D_P F_2 D_P$, we have \begin{align} B_{11} &=A_{11}^* P_1, \quad B_{12} =A_{21}^* P_2, \label{eqn3.3}\\ B_{21} &=A_{12}^* P_1, \quad B_{22} =A_{22}^* P_2 + D_{P_2} F_{2}' D_{P_2} \label{eqn4.3}. \end{align} Since $AP=PQA$, we have \begin{align} A_{11} P_1 &=P_1Q_1A_{11}, \quad A_{12}P_2=P_1Q_1A_{12}, \label{eqn5.3} \\ A_{21} P_1 &=P_2Q_2A_{21}, \quad A_{22} P_2=P_2Q_2A_{22} \label{eqn6.3}. \end{align} Also, from $BQP=PB,$ we have \begin{align} B_{11} Q_1P_1 &=P_1 B_{11}, \quad B_{12} Q_2P_2=P_1 B_{12}, \label{eqn7.3} \\ B_{21} Q_1P_1 &=P_2 B_{21}, \quad B_{22}Q_2 P_2=P_2 B_{22}. \label{eqn8.3} \end{align} Since $P$ and $Q$ commute, using (\ref{eqn3.3}), (\ref{eqn2.3}) and (\ref{eqn6.3}), we have \begin{equation}\label{eqn9.3} P_2 P_2^* A_{21}=P_2 B_{12}^*=P_2Q_2A_{21}P_1^*=A_{21}. \end{equation} Using (\ref{eqn6.3}), (\ref{eqn3.3}) and (\ref{eqn2.3}), we have \begin{equation}\label{eqn10.3} P_2^* P_2 A_{21}=Q_2^*P_2^*A_{21}P_1=Q_2^*B_{12}^*P_1=A_{21}. \end{equation} By (\ref{eqn6.3}), it follows that the range of $A_{21}$ is invariant under $P_2Q_2$, and that
  \[
   P_2^*Q_2^* A_{21}=P_2^* P_2 A_{21}P_1^*=A_{21}P_1^*.
  \]
  Therefore, the range of $A_{21}$ is invariant under $(P_2Q_2)^*$. By (\ref{eqn9.3}) \& (\ref{eqn10.3}), it follows that $P_2Q_2$ acts as a unitary on the range of $A_{21}$. Since $P_2$ is a c.n.u. contraction and $Q_2$ is a unitary, Lemma \ref{TUcnu} implies that $P_2Q_2$ is a c.n.u. contraction. Therefore, $A_{21}=0.$ \medskip Next, we show that $B_{21}=0$. From (\ref{eqn8.3}), (\ref{eqn1.3}) and (\ref{eqn4.3}), we have \begin{equation}\label{eqn11.3} P_2^*P_2B_{21}=P_2^*B_{21}Q_1P_1=A_{12}^*Q_1^*Q_1P_1=A_{12}^*P_1=B_{21}. \end{equation} Similarly, from (\ref{eqn1.3}), (\ref{eqn4.3}) and (\ref{eqn8.3}), we have \begin{equation}\label{eqn12.3} P_2P_2^*B_{21}=P_2A_{12}^*Q_1^*=P_2B_{21}P_1^*Q_1^*=B_{21}. \end{equation} It follows from (\ref{eqn8.3}) that \begin{equation}\label{eqn13.3} P_2B_{21}=B_{21}Q_1P_1 \end{equation} and so, \begin{equation}\label{eqn14.3} P_2^*B_{21}P_1=P_2^*P_2B_{21}Q_1^*=B_{21}Q_1^*. \end{equation} Thus the range of $B_{21}$ is invariant under $P_2$ and $P_2^*$. It follows from (\ref{eqn11.3}) \& (\ref{eqn12.3}) that $P_2$ acts as a unitary of the range of $B_{21}$. Since $P_2$ is a c.n.u. contraction, we get that $B_{21}=0$. Consequently, $A_{12}=0=B_{12}$ by the virtue of (\ref{eqn1.3}) and (\ref{eqn3.3}). The proof is now complete. \end{proof}	
Finally, we need the following terminology to present the main result of this Section.
	
\begin{defn}
		A pair of contractions $(A, B)$ (not necessarily commuting) acting on a Hilbert space $\HS$ is said to be a \textit{totally c.n.u. pair} if whenever there is a proper closed subspace $\mathcal{L}$ of $\HS$ which reduces both $A$ and $B$ to unitaries, then $\mathcal{L}=\{0\}$.  
	\end{defn}
An example of a totally c.n.u. pair of contractions will be appropriate here.
\begin{eg}
Let $U$ be a unitary on a Hilbert space $\mathcal E$ and let $T_1, T_2 \in \mathcal{B}(\ell^2(\mathcal E))$ be given by
\[
T_1(a_0, a_1, a_2, \dotsc) = (0, a_0,Ua_1,U^2a_2, \dotsc) \quad \text{and} \quad T_2(a_0, a_1, a_2, \dotsc) = (0, a_0, a_1, a_2, \dotsc) .
\]
Let us define $Q : \ell^2(\mathcal E) \to \ell^2(\mathcal E), Q(a_0, a_1, \dotsc) = (Ua_0,Ua_1, \dotsc)$. Recall from Example \ref{eg502} that $(T_1, T_2)$ is a pair of pure isometries. It is not difficult to see that 
\[
QT_1=T_1Q, \quad QT_2=T_2Q, \quad T_1T_2=QT_2T_1, \quad T_1T_2=T_2QT_1 \quad \text{and} \quad T_1T_2=T_2T_1Q.
\]
Let $\alpha_1, \alpha_2$ be scalars satisfying $0<|\alpha_1|, |\alpha_2|<1$ and let $(A_1, A_2)=(\alpha_1T_1, \alpha_2T_2)$. It follows that $\|A_jh\|=|\alpha_j| \ \|T_jh\|=|\alpha_j| \|h\|<\|h\|$ for every non-zero $h \in \ell^2(\mathcal E)$ and $j=1, 2$. Therefore, $(A_1, A_2)$ is a totally c.n.u. pair of contractions that satisfies $A_1A_2=QA_2A_1=A_2QA_1=A_2A_1Q$ and $A_1A_2Q=QA_1A_2$. \qed
\end{eg}	
Below we present the main result of this Section.
	
	\begin{thm}\label{thm712}
		Let $(A,B)$ be a pair of contractions acting on a Hilbert space $\HS$ and let $Q\in \mathcal B(\HS)$ be a unitary such that any of the following relations holds: 
		\[
		AB=QBA, \quad AB=BQA \quad \text{or} \quad AB=BAQ.
		\]
Let $P=AB$ and let $QP=PQ$. If $P=P_1 \oplus P_2$ is the canonical decomposition of $P$ with respect to the orthogonal decomposition $\HS=\HS_1 \oplus \HS_2$, then the following hold:  
		\begin{enumerate}
			\item $\mathcal{H}_1, \HS_2$ are joint reducing subspaces for $A, B$ and $Q$;
			\item $\mathcal{H}_1$ is the largest closed reducing subspace for $A,B$ such that both $A|_{\mathcal{H}_1}, B|_{\mathcal{H}_1}$ are unitaries;
			\item $(A|_{\mathcal{H}_2}, B|_{\mathcal{H}_2})$ is a totally c.n.u. pair of contractions;
			\end{enumerate}
	\end{thm}
	
	\begin{proof} 
		It follows from Theorem \ref{reducingII} that $\mathcal{H}_1, \mathcal{H}_2$ reduce $A, B$ and $Q$. We have by Lemma \ref{P_unitary} that  $A|_{\mathcal{H}_1}, B|_{\mathcal{H}_1}$ are unitaries. Let $\mathcal{L}$ be a closed subspace of $\HS$ that reduces $A$ and $B$ to unitaries. Since $P=AB$, it is immediate that $\mathcal{L}$ is a reducing subspace for $P$. By Lemma \ref{P_unitary}, $P|_\mathcal{L}$ is a unitary and so, $\mathcal{L} \subseteq \HS_1$ as $\HS_1$ is the largest subspace of $\HS$ that reduces $P$ to a unitary. Thus, $\mathcal{H}_1$ is also the largest closed reducing subspace for $A,B$ such that both $A|_{\mathcal{H}_1}, B|_{\mathcal{H}_1}$ are unitaries. Let $\mathcal{L} \subseteq \HS_2$ be a proper closed joint reducing subspace for $A, B$ such that $A|_\mathcal{L}, B|_\mathcal{L}$ are unitaries. It follows from Lemma \ref{P_unitary} that $P|_\mathcal{L}$ is a unitary. Therefore, $\mathcal{L} \subseteq \HS_1 \cap \HS_2=\{0\}$ and so, $(A|_{\mathcal{H}_2}, B|_{\mathcal{H}_2})$ is a totally c.n.u. pair of contractions. The proof is now complete. 
	\end{proof}

	\vspace{0.2cm}
	
	\noindent \textbf{Funding.} The first named author is supported by the Seed Grant of IIT Bombay, the CDPA and the `Core Research Grant' with Award No. CRG/2023/005223 of Science and Engineering Research Board (SERB), India. The second named author is supported by the Prime Minister's Research Fellowship (PMRF ID 1300140), Government of India.
	
	\section{Data Availability Statement}

\noindent (1) Data sharing is not applicable to this article as no datasets were generated or analysed during the current study.

\smallskip

\noindent (2) In case any datasets are generated during and/or analysed during the current study, they must be available from the corresponding author on reasonable request.


\section{Declarations}


\noindent \textbf{Ethical Approval.} This declaration is not applicable.

\smallskip

\noindent \textbf{Competing interests.} There are no competing interests.

\smallskip

\noindent \textbf{Authors' contributions.} This declaration is not applicable.

	\end{document}